\documentclass[oneside,openany,a4paper]{amsart}
\usepackage[a4paper, margin=1in]{geometry}
\usepackage{amssymb}
\usepackage[centertags]{amsmath}
\usepackage{amsthm}
\usepackage{amsfonts}
\usepackage{eucal}
\usepackage{mathrsfs}
\usepackage{enumerate}

\usepackage[all,cmtip]{xy}
\usepackage{graphicx, eepic}

\usepackage{lmodern}
\usepackage[utf8]{inputenc}
\usepackage[T1]{fontenc}
\usepackage{pdfrender, xcolor}

\usepackage{bbm}
\usepackage{pdfpages}
\usepackage{titletoc}
\usepackage{booktabs}
\usepackage[shortlabels]{enumitem}
\usepackage{mathtools}
\usepackage{thmtools}
\usepackage{pgfplots,pgfplotstable}

\usepackage{epstopdf}
\usepackage{epsfig}
\usepackage{microtype}
\usepackage[ruled,lined]{algorithm2e}

\usepackage{hyperref}
\hypersetup{colorlinks}
\hypersetup{                
    unicode=false,          
    pdftoolbar=true,        
    pdfmenubar=true,        
    pdffitwindow=false,     
    pdfstartview={FitH},    
    pdftitle={My title},    
    pdfauthor={Author},     
    pdfsubject={Subject},   
    pdfcreator={Creator},   
    pdfproducer={Producer}, 
    pdfkeywords={keyword1} {key2} {key3}, 
    pdfnewwindow=true,      
    colorlinks=true,       
    linkcolor=red,          
    citecolor=magenta,        
    filecolor=violet,      
    urlcolor=blue      
}
\usepackage{multirow, url}
\usepackage{textcomp}
\DeclareMathAlphabet{\cmcal}{OMS}{cmsy}{m}{n}
\newtheoremstyle{thm}
  {3pt}
  {3pt}
  {\em}
  {0pt}
  {\bfseries}
  {}
  {5pt}
  {}
\newtheoremstyle{rem}
  {3pt}
  {3pt}
  {}
  {0pt}
  {\bfseries}
  {.}
  {5pt}
  {}

\newtheorem{thm}{Theorem}[section]
\newtheorem{cor}[thm]{Corollary}
\newtheorem{lem}[thm]{Lemma}
\newtheorem{prop}[thm]{Proposition}

\newtheorem{conj}[thm]{Conjecture}

\theoremstyle{definition}
\newtheorem{defn}[thm]{Definition}

\theoremstyle{rem}
\newtheorem{rem}[thm]{{Remark}}

\numberwithin{equation}{section} \numberwithin{table}{section}
	
\newtheorem*{thm*}{Theorem}
\newtheorem*{rem*}{Remark}
\newtheorem*{rems*}{Remarks}
\newtheorem*{exam*}{Example}
\newtheorem*{exams*}{Examples}

\makeatletter
\newcommand{\neutralize}[1]{\expandafter\let\csname c@#1\endcsname\count@}
\makeatother

\usepackage{enumerate}

\def\bos#1{{\mathbf{#1}}}

 \newcommand{\res}{\operatorname{Res}}

  \newcommand{\nc}{\newcommand}
  \newcommand{\be}{\begin{eqnarray*}}
  \newcommand{\ee}{\end{eqnarray*}}
  \newcommand{\bea}{\begin{eqnarray}}
  \newcommand{\eea}{\end{eqnarray}}
  \newcommand{\bs}{\begin{split}}
  \newcommand{\es}{\end{split}}
  \newcommand{\bal}{\begin{align}}
  \newcommand{\eal}{\end{align}}

   \nc{\bei}{\begin{itemize}}
   \nc{\eei}{\end{itemize}}
   \nc{\bee}{\begin{enumerate}}
   \nc{\eee}{\end{enumerate}}
   \nc{\bet}{\begin{thm}}
   \nc{\eet}{\end{thm}}
   \nc{\bed}{\begin{defn}}
   \nc{\eed}{\end{defn}}
   \nc{\bel}{\begin{lem}}
   \nc{\eel}{\end{lem}}
   \nc{\bep}{\begin{prop}}
   \nc{\eep}{\end{prop}}
   \nc{\bec}{\begin{corollary}}
   \nc{\eec}{\end{corollary}}
   \nc{\ber}{\begin{rem}}
   \nc{\eer}{\end{rem}}
   \nc{\beex}{\begin{example}}
   \nc{\eeex}{\end{example}}

   \nc{\bpm}{\begin{pmatrix}}
   \nc{\epm}{\end{pmatrix}}
   \nc{\bspm}{\left(\begin{smallmatrix}}
   \nc{\espm}{\end{smallmatrix}\right)}

\newcommand{\cC}{\mathcal{C}}

\newcommand{\cO}{\mathcal{O}}

\newcommand{\cU}{\mathcal{U}}

\newcommand{\bC}{\mathbb{C}}

\newcommand{\bH}{\mathbb{H}}

\newcommand{\bQ}{\mathbb{Q}}

\newcommand{\BP}{\mathbf{P}}

\nc{\frf}{\mathfrak{f}} 

\nc{\frs}{\mathfrak{s}}  
\nc{\frt}{\mathfrak{t}} 
\nc{\fru}{\mathfrak{u}}
  
\nc{\lsl}{\mathfrak{sl}}
\nc{\lgl}{\mathfrak{gl}}

\nc{\upsi}{\underline{\psi}}
\nc{\uchi}{\underline{\chi}}

\newcommand{\lra}{\longrightarrow}    

\nc{\surjto}{\twoheadrightarrow}
\nc{\ts}{\times}
\nc{\ds}{\displaystyle}
\nc{\nd}{\noindent}  
\nc{\ud}{\underline}
\nc{\ov}{\overline}
\nc{\maplra}[1]{\buildrel #1 \over \lra}
\nc{\mapto}[1]{\buildrel #1 \over \to}
\nc{\setb}[1]{\{  #1\}}

 \nc{\cHom}{\mathcal{H}om}

\nc{\cdruur}[8] {\begin{CD} 
#1 @>#2>> #3\\ 
@AA#4A @AA#5A\\ 
#6 @>#7>> #8 
\end{CD} }
\nc{\cdrddr}[8] {\begin{CD} 
#1 @>#2>> #3\\ 
@VV#4V @VV#5V\\ 
#6 @>#7>> #8 
\end{CD} }

\nc{\dia}[8]{\xymatrix{ 
&#1 \ar@{-}[ld]_{#2} \ar@{-}[rd]^{#3} \\
#4 \ar@{-}[rd]_{#6} & &#5 \ar@{-}[ld]^{#7}\\ 
&#8} }

\nc{\diam}[9]{\xymatrix{ 
&#1 \ar@{-}[ld]_{#2}  \ar@{-}[d]^{#3} \ar@{-}[rd]^{#4} \\
#5 \ar@{-}[rd]_{#8}     & #6 \ar@{-}[d]_{#9}      & #7   \ar@{-}[ld]^{2} \\
& \bQ} } 

\nc{\sumn}[2][n]{#2_{1} +#2_{2}+ \cdots + #2_{#1}}
\nc{\poly}[3][n]{#2_{#1}#3^{#1} +#2_{#1-1}#3^{#1-1}  \cdots + #2_{1} #3+ #2_0}
\nc{\dpoly}[3][n]{#1#2_{#1}#3^{#1-1} +(#1-1)#2_{#1-1}#3^{#1-1}  \cdots +2 #2_{2} #3+ #2_1}
\nc{\mpoly}[3][n]{#3^{#1} +#2_{#1-1}#3^{#1-1}  \cdots + #2_{1} #3+ #2_0}

\nc{\vpar}[4]{    \left \{ \begin{array}{cc} #1 & \textrm{if } #2, \\
&\\
#3 & \textrm{if } #4. 
\end{array}\right. }

\nc{\ary}[5]{#1: \left\{ \begin{array}{ll} #2 &\mapsto #3 \\ #4 &\mapsto #5 \end{array} \right.}
 \nc{\bedm}{\begin{displaymath}}
 \nc{\eedm}{\end{displaymath}}
 \nc{\art}{\hbox{\bf Art}^\Z}
 \nc{\bvx}{\bos{B\!\!V}_{\! \!X}}

\newcommand{\pmat}{\left(\begin{matrix}}   
\newcommand{\epmat}{\end{matrix}\right)}   
\newcommand{\psmat}{\left(\begin{smallmatrix}}    
\newcommand{\epsmat}{\end{smallmatrix}\right)}
\nc{\twotwo}[4]{\pmat #1 & #2 \\ #3 & #4 \epmat}
\nc{\thrthr}[9]{\pmat #1 & #2 & #3 \\ #4 & #5 & #6 \\ #7 & #8 & #9 \epmat}
\nc{\stwotwo}[4]{\psmat #1 & #2 \\ #3 & #4 \epsmat}
\nc{\sthrthr}[9]{\psmat #1 & #2 & #3 \\ #4 & #5 & #6 \\ #7 & #8 & #9 \epsmat}

\def\eqalign#1{\null\,\vcenter{\openup\jot\m@th
\ialign{\strut\hfil$\displaystyle{##}$&$\displaystyle{{}##}$\hfil
\crcr#1\crcr}}\,}

\def\eqn#1#2{
\xdef #1{(\nsecsym\the\meqno)}
\global\advance\meqno by1
$$#2\eqno#1\eqlabeL#1
$$}

\def\a{\alpha}

\def\d{\delta}

\def\o{\omega}  \def\O{\Omega}

\def\Z{\mathbb{Z}}

\def\mod{\hbox{ }mod\hbox{ }}

\catcode`\@=11 

\global\newcount\nsecno \global\nsecno=0
\global\newcount\meqno \global\meqno=1
\def\newsec#1{\global\advance\nsecno by1
\eqnres@t
\section{#1}}
\def\eqnres@t{\xdef\nsecsym{\the\nsecno.}\global\meqno=1}
\def\sequentialequations{\def\eqnres@t{\bigbreak}}\xdef\nsecsym{}

\def\draftmode{\message{ DRAFTMODE }

{\count255=\time\divide\count255 by 60 \xdef\hourmin{\number\count255}
\multiply\count255 by-60\advance\count255 by\time
\xdef\hourmin{\hourmin:\ifnum\count255<10 0\fi\the\count255}}}
\def\nolabels{\def\wrlabeL##1{}\def\eqlabeL##1{}\def\reflabeL##1{}}
\def\writelabels{\def\wrlabeL##1{\leavevmode\vadjust{\rlap{\smash%
{\line{{\escapechar=` \hfill\rlap{\tt\hskip.03in\string##1}}}}}}}%
\def\eqlabeL##1{{\escapechar-1\rlap{\tt\hskip.05in\string##1}}}%
\def\reflabeL##1{\noexpand\llap{\noexpand\sevenrm\string\string\string##1}
}}

\nolabels

\def\eqn#1#2{
\xdef #1{(\nsecsym\the\meqno)}
\global\advance\meqno by1
$$#2\eqno#1\eqlabeL#1
$$}

\def\eqalign#1{\null\,\vcenter{\openup\jot\m@th
\ialign{\strut\hfil$\displaystyle{##}$&$\displaystyle{{}##}$\hfil
\crcr#1\crcr}}\,}

\vspace{.2in}

   \nc{\hr}{[\![\hbar]\!]}

\catcode`\@=01 

\nc{\bt}{\mathbf{t}}
\draftmode

\begin{document}

\title[ An algorithm for a Massey triple product of a smooth projective plane curve ]{An algorithm for a Massey triple product of a smooth projective plane curve}
%
\author{Younggi Lee}
\email{yglee@postech.ac.kr}
\author{Jeehoon Park}
\email{jeehoonpark@postech.ac.kr}
\author{Junyeong Park}
\email{junyeongp@gmail.com}
\author{Jaehyun Yim}
\email{yimjaehyun@postech.ac.kr}
\address{
Department of Mathematics, POSTECH,
77 Cheongam-Ro, Namgu,	
Pohang, Gyeongbuk, 37673,
South Korea.}
%


\begin{abstract}
We provide an explicit algorithm to compute a Massey triple product relative to a defining system for a smooth projective plane curve $X$ defined by a homogeneous polynomial $G(\ud x)$ over a field. 
The main idea is to use the description (due to Carlson and Griffiths) of the cup product for $H^1(X,\bC)$ in terms of the multiplications inside the Jacobian ring of $G(\ud x)$ and the Cech-deRham complex of $X$. 
Our algorithm gives a criterion whether a Massey triple product vanishes or not in $H^2(X)$ under a particular non-trivial defining system of the Massey triple product and thus can be viewed as a generalization of the vanishing criterion of the cup product in $H^2(X)$ of Carlson and Griffiths. Based on our algorithm, we provide explicit numerical examples by running the computer program.
\end{abstract}

\maketitle
\tableofcontents


\section{Introduction}

The Massey product is an important invariant in topology and appears in several areas of mathematics as a higher order structure which captures more refined informations than the usual cup product of the cohomology. 

Let $M$ be a topological manifold. Assume that there are three classes $x_i \in H^{m_i}(M,\bC),i=0,1,2$ such that $x_0 \cup x_1 =0=x_1\cup x_2$. Choose cochain representatives $X_i$ for the $x_i$. Since
$X_0 \cup X_1$ is a coboundary, choose $X_{01}$ so that $d(X_{01}) =X_0 \cup X_1.$ Similarly choose $X_{12}$ so that $d(X_{12}) = X_1 \cup X_2.$ Form the cochain
$$
X_{01} \cup X_2 -(-1)^{m_0}X_0 \cup X_{12}
$$
which becomes a cocycle. The choices of $X_{01}$ and of $X_{12}$ are not unique. Precisely either can be altered by a cocycle. Define the ideal
$$
J^\bullet_{{x_0},{x_2}} =x_0 \cup H^{\bullet-m_0}(M,\bC)+H^{\bullet-m_2}(M,\bC)\cup x_2 \subseteq H^\bullet(M,\bC)
$$ 
The cocycle $X_{01} \cup X_2 -(-1)^{m_0}X_0 \cup X_{12}$ can be altered by (a cochain representative of) any element in 
$J^{m_0+m_1+m_2-1}_{x_0,x_2}$.
For any $x_1 \in H^{m_2}(M,\bC)$ such that $x_0 \cup x_1 =0=x_1\cup x_2$,
the Massey triple product
$$
\langle x_0; x_1; x_2 \rangle \in H^{m_0+m_1+m_2-1}(M,\bC)/J^{m_0+m_1+m_2-1}_{{x_0},{x_2}} 
$$
is defined to be the class of $X_{01} \cup X_2 -(-1)^{m_0}X_0 \cup X_{12}$ in this quotient group. 

When $M$ is a compact Riemann surface (the main object of interest in the current paper) and $m_0=m_1=m_2=1$, then $H^{2}(M,\bC)/J^{2}_{{x_0},{x_2}}=0$ for nonzero $x_0, x_2 \in H^1(M,\bC)$ by the non-degeneracy of the cup product pairing for $H^1(M,\bC)$ and the Massey triple product vanishes.
In general, when $M$ is a compact K\"ahler manifold, then the Massey triple product is known to vanish by the result
of \cite{DGMS}.
 But in this article we will be interested in computing 
 \textit{the lifted Massey triple product} which is defined to be
the cohomology class (depending on $X_{01}$ and $X_{12}$)
$$
[ x_0, x_1, x_2]:=[X_{01} \cup X_2 -(-1)^{m_0}X_0 \cup X_{12}] \in H^{m_0+m_1+m_2-1}(M,\bC),
$$
for \textit{a particular explicit choice} of $X_{01}$ and $X_{12}$.
In other words, we will provide an algorithm for the value of the Massey triple product relative to a certain defining system (see Definition \ref{mase} below).

\begin{defn} \label{mase}
The Massey product $[f_0, . . . , f_n]$ is said to be defined if there is an array $A$ of Cech-de Rham 1-cochains
\begin{eqnarray*}
A=\{ F_{i_1 \cdots i_k} \in \cC_0^1 \oplus \cC_1^0    :   i_{p+1} = i_p + 1\ (\forall p) \}_{0 \leq k \leq n-1}
\end{eqnarray*}
such that $F_i$ is a 1-cocycle representing $f_i$ for $0\leq i  \leq n,$ and
$$
D(F_{i_1\cdots i_k})=F_{i_1\cdots i_{k-1}} \bullet F_{i_k} +F_{i_1}\cdots i_{k-2} \bullet F_{i_{k-1}i_k} + \cdots + F_{i_1} \bullet F_{i_2 \cdots i_k},
$$
for $0 \leq k \leq n-1, i_{p+1}=i_p +1 \ (\forall p)$.
\end{defn}
An array $A$ is called a defining system for $[f_0, \cdots, f_n]$. The value of the Massey $(n+1)-$th product relative to a defining system $A$, denoted by $[f_0,\cdots, f_n]_A$, is then defined to be the cohomology class of $\bH^2(X_G)$ represented by the 2-cocycle
$$
c(A):=F_{0\cdots n-1}\bullet F_n + F_{0 \cdots n-2}\bullet F_{n-1 n} +\cdots + F_0 \bullet F_{1\cdots n}.
$$

We will provide an explicit defining system $A$ for the Massey triple product of a projective smooth plane curve (based on Carlson-Griffiths' Jacobian vanishing criterion of the cup product and computations in Cech-deRham complex with respect to a Jacobian covering) and an algorithmic formula (see the main theorem \ref{mt}) for the $\bC$-valued Massey triple product relative to $A$. 



%

%
%

Now we explain a more precise set up of this paper. Let $X_G$ be a smooth projective curve in $\BP^2$ defined by a homogeneous polynomial $G(\ud x) \in \Bbbk[\ud x]=\Bbbk[x_0, x_1, x_2]$
over a field $\Bbbk$. Let $\BP^2_\Bbbk$ be a projective 2-space over $\Bbbk$.
Assume that $\Bbbk$ is the complex field $\bC$ for the simplicity of presentation.\footnote{Since the theory is algebraic, our algorithm has a chance to be applied to more general fields, e.g. the $p$-adic Dwork cohomology defined over a $p$-adic field.}
Let $X=X_G(\bC)$ (respectively, $\BP^2=\BP^2_\bC(\bC)$) and  be the complex analytic space (the compact Riemann surface)
associated to $X_G$ (respectively, $\BP^2_\bC$). 
For each $i =0,1,2$, set
$$
G_i := \frac{\partial G(\ud x)}{\partial x_i}, \quad \cU_i = \{ \ud x \in \BP^2 : G_i(\ud x) \neq 0\}=\BP^2 \setminus X_{G_i}(\bC).
$$
Then the Griffiths theorem, \cite{Gr69}, asserts a canonical isomorphism between the middle-dimensional singular cohomology $H^1(X,\bC)$ and the degree $\deg G-3$ and $2\deg G-3$ part of the quotient $\bC$-vector space $\bC[\ud x]/J_G$ where $J_G$ is the Jacobian ideal of $\bC[\ud x]$ generated by $G_0, G_1, G_2$.
If we let $U_0,U_1 \in \bC[\ud x]$ correspond to singular cocycles $W_0,W_1$ (whose cohomology class $[W_i] \in H^1(X,\bC)$) respectively under the Griffiths isomorphism, then the theorem 2, \cite{CG}, says that $[W_0 \cup W_1]$ vanishes in $H^2(X,\bC)$ if and only if
$$
\text{
$U_0 U_1$ belongs to the Jacobian ideal $J_G= \langle G_0(\ud x), G_1(\ud x), G_2(\ud x) \rangle$.
}
$$

Our main result will be in the following shape: if $U_0, U_1, U_2 \in \bC[\ud x]$ correspond to
$W_0,W_1,W_2$ such that $U_0U_1, U_1U_2 \in J_G$ (i.e. $U_0,U_1,U_2$ satisfies the condition to define a Massey triple product), then the Massey triple product $[W_0,W_1,W_2]$ under a particular non-trivial defining system (i.e. a choice of $W_{01}$ and $W_{12}$: see Proposition \ref{wtwo}) vanishes in $H^2(X,\bC)$ 
if and only if \par
\vspace{0.5em}
a certain explicit quantity (see the quantity $A-B$ in Theorem \ref{mt}) 
belongs to the ideal generated by $G_0(\ud x)^2, G_1(\ud x)^2, G_2(\ud x)^2$.

\begin{rem}
The similarity of vanishing criterions of the cup product and the triple Massey product suggests us to conjecture that the $n$-Massey product vanishes (under a particular non-trivial defining system) if and only if some explicit quantity belongs to the ideal generated by $G_0(\ud x)^{n-1}, G_1(\ud x)^{n-1}, G_2(\ud x)^{n-1}$. But due to technical difficulties arising in the computation of higher ($n\geq 4$) Massey products, we were not able to provide such general vanishing criterion.
\end{rem}

We consider the Cech-deRham double complex of $X$
for the open cover $\cU =\{ \cU_i\}_{i=0,1,2}$:
$$
\cC^{p}_{q} := \cC^q (\cU |_X, \Omega_X^p),
$$
where $\Omega_X^p$ is the sheaf of \textit{holomorphic} differential $p$-forms on $X$.
We will compute the Massey triple product for the first degree cohomology of the associated single complex (so called, the hypercohomology) with respect to the total differential
$$
D = d+(-1)^p \d,
$$

where $d$ is the exterior differential and $\d$ is the Cech differential.
Remark that the smoothness of the $X_G$ is essential, since otherwise $\cU$ do not cover $\BP^2$. We shall refer to $\cU$ as the Jacobian cover of $\BP^2$ relative to $G(\ud x)$.

Let $\bH^\bullet(X_G)=\bigoplus_{p+q=\bullet} H^p(\cU|_X, \O_X^q)$ be the hypercohomology of the Cech-deRham complex of $X_G$ with respect to the Jacobian cover $\cU$.
Note that the hypercohomology $\bH^\bullet(X_G)$ of this associated single complex is isomorphic to the singular cohomology $H^\bullet(X, \bC)$.
Let us briefly recall the natural twisted product structure on the Cech-deRham complex. Let $\wedge$ denote the natural product
$$
\cC^p_q \times \cC_s^r \to \cC^{p+r}_{q+s}.
$$
given by the "front q-face, back s-face" rule, followed by exterior multiplication of forms. Then the twisted product
 $$
 \a_q^p\bullet \a_s^r:= (-1)^{qr} \a_q^p \wedge \a_s^r 
 $$
is skew-commutative and satisfies the Leibnitz rule:
 \begin{eqnarray*}
  \a_s^r \bullet \a_q^p &=&(-1)^{(p+q)(s+r)} \a_q^p\bullet \a_s^r, \\
 D( \a_q^p\bullet \a_s^r) &=& D(\a_q^p)\bullet \a_s^r + (-1)^{p+q}\a_q^p \bullet D(\a_s^r).
 \end{eqnarray*}
 
%
 
Because the twisted product reduces to exterior multiplication of forms, it represents the topological cup-product on the level of hypercohomology. Therefore, the Massey triple product for the hypercohomology $\bH^1(X)$ gives the Massey triple product for $H^1(X, \bC)$.

Let $w_0, w_1, w_2 \in \bH^1(X_G)$ such that $w_0 \cup w_1 = w_1 \cup w_2=0$ in $\bH^2(X_G)$.
More precisely, we consider $W_0, W_2 \in \cC_1^0$ and $W_1 \in \cC_0^1$ with the associated cohomology class $[W_i] =w_i$ such that  
$$D(W_{01}) = W_0 \bullet W_1, \quad D (W_{12}) = W_1 \bullet W_2 \quad \text{in } \cC_1^1
$$ 
for some $W_{01}, W_{12} \in \cC_0^1$. Note that $[W_0 \bullet W_1]=w_0 \cup w_1$ and $[W_1 \bullet W_2]=w_1 \cup w_2$.
Then we can form the Massey triple product:
$$
[ w_0, w_1, w_2 ] := [W_0 \bullet W_{12} + W_{01} \bullet W_2] \in  H^1(\cU|_X, \O_{X}^1).
$$
Since there is an isomorphism $\theta: H^1(\cU|_X, \O_{X}^1) \buildrel \simeq \over \longrightarrow \bC$, we can define a $\bC$-valued Massey triple product:
$$
\langle w_0, w_1, w_2 \rangle := \theta([ w_0, w_1, w_2 ]) \in \bC.
$$

We will find explicit formulas for a non-trivial defining system $A=\{W_0, W_1, W_2, W_{01}, W_{12}\}$ and the corresponding $\bC$-valued Massey triple product, which can be implemented in computer algebra system. We will also show (see Corollary \ref{indep}) that vanishing of the Massey triple product does not depend on choices of cohomology representatives of $[W_0],[W_1],[W_2]$ under our particular non-trivial defining system.


\subsection{Acknowledgement}
Jeehoon Park would like to thank Minhyong Kim for a useful related discussion and this project was motivated from the conversation with him.
The work of Jeehoon Park was partially supported by BRL (Basic Research Lab) through the NRF (National Research Foundation) of South Korea (NRF-2018R1A4A1023590).

\section{Main computations}

\subsection{Griffiths' description for $H^1(X,\bC)$}
We introduce a new variable $y$ and consider the polynomial ring 
$\Bbbk[x_0, x_1, x_2, y]$. Let us define the charge (additive) grading on $\Bbbk[x_0,x_1,x_2,y]$:
$$
ch(x_i) =1, ch(y)=-\deg G.
$$
Then $\Bbbk[x_0,x_1,x_2,y]$ has charge decomposition:
$$
\Bbbk[x_0,x_1,x_2,y] =\bigoplus_{ch} \Bbbk[x_0,x_1,x_2,y]_{ch} .
$$
Let $S(y, \ud x) = y G(\ud x)$ and $c_X:= \deg G- 3$. 
Let $J_S$ be the Jacobian ideal of $\Bbbk[x_0, x_1, x_2, y]$ generated by
$\frac{\partial S(\ud x) }{\partial x_i}, \ i = 0,1,2$, and $\frac{\partial S(\ud x) }{\partial y}$
Then there is a canonical isomorphism due to Griffiths,  \cite{Gr69},
(for example, see the Proposition 4.10 and section 4 of \cite{PP16} for a detailed explanation)
\begin{eqnarray*} 
\Bbbk[x_0,x_1,x_2,y]_{c_{X}} /J_S \cap \Bbbk[x_0,x_1,x_2,y]_{c_{X}}\simeq H^2(\BP^2\setminus X) \xrightarrow{\res_X} H^1(X,\bC),
\end{eqnarray*}
where the first map is given by 
$$
y^{k-1} u(\ud x) \mapsto (-1)^{k-1} (k-1)! \frac{u(\ud x)}{ G(\ud x)^k} \cdot( x_0 dx_1 \wedge dx_2 - x_1 dx_0 \wedge dx_2 + x_2 dx_0 \wedge dx_1), \quad k \geq 1
$$
and $\res_X$ is the residue map (e.g. see p 298, \cite{PP16} for its precise description).
Let $J_G$ be the Jacobian ideal of $\Bbbk[\ud x]=\Bbbk[x_0, x_1, x_2]$ generated by
$\frac{\partial G(\ud x) }{\partial x_i}, \ i = 0,1,2$. 
Then the inclusion $\Bbbk[\ud x] \hookrightarrow \Bbbk[\ud x, y]$ induces an isomorphism
\begin{eqnarray}\label{Gisom}
 \bigoplus_{m\geq 0}\Bbbk[\ud x]_{m\deg G +c_X }/ J_G \cap  \Bbbk[\ud x]_{m\deg G+c_X} \simeq \Bbbk[\ud x,y]_{c_{X}} /J_S \cap \Bbbk[\ud x,y]_{c_{X}},
\end{eqnarray}
where $\Bbbk[\ud x]_\a$ is the submodule of $\Bbbk[\ud x]$ consisting of homogeneous polynomials of degree $\a$.
Note that $\Bbbk[\ud x]_{m\deg G +c_X }=\Bbbk[\ud x]_{(m+1)\deg G -3}$.
Let us consider the Hodge decomposition $H^1(X,\bC) \simeq H^0(X, \Omega_X^1) \oplus
H^1(X, \Omega_X^0).$
Then it is known (the pole order filtration with respect to $G(\ud x)$ is sent to the Hodge filtration under $\res_X$ due to Griffiths, Theorem 8.3, \cite{Gr69}) that 
\begin{eqnarray*}
\Bbbk[\ud x]_{\deg G -3 }/ J_G \cap\Bbbk[\ud x]_{\deg G-3}  &\simeq& H^0(X, \Omega_X^1),\\
 \bigoplus_{1\leq m\leq 2}\Bbbk[\ud x]_{m\deg G -3 }/ J_G \cap \Bbbk[\ud x]_{m\deg G -3 }
&\simeq& H^0(X, \Omega_X^1) \oplus
H^1(X, \Omega_X^0),\\
\Bbbk[\ud x]_{m\deg G -3 }/ J_G \cap\Bbbk[\ud x]_{m\deg G-3} &=&0, \quad \text{if } m \geq 3 \quad \text{ (Macaulay's theorem)},
\end{eqnarray*}
under the isomorphism from $ \bigoplus_{m\geq 0}\Bbbk[\ud x]_{(m+1)\deg G -3 }/ J_G \cap  \bigoplus_{m\geq 0}\Bbbk[\ud x]_{(m+1)\deg G -3} $ to $H^1(X,\bC)$.
Consider the multiplication map followed by the projection to degree $3\deg G-6$-part:
\small{
$$
 \bigoplus_{1\leq m\leq 2} \Bbbk[\ud x]_{m\deg G -3 }/ J_G \cap\Bbbk[\ud x]_{m\deg G-3} \times
 \bigoplus_{1\leq m\leq 2}\Bbbk[\ud x]_{m\deg G -3 }/ J_G \cap\Bbbk[\ud x]_{m\deg G-3}
\to
\Bbbk[\ud x]_{3\deg G -6 }/ J_G \cap\Bbbk[\ud x]_{3\deg G-6}.
$$
}
In \cite{CG}, it is shown that this multiplication matches with the cup product
$$
 H^1(X,\bC) \times H^1(X,\bC)  \to H^2(X, \bC) \simeq \bC
$$
under the above isomorphism and the isomorphism (the corollary (iv), p 19, \cite{CG})
$$
\Bbbk[\ud x]_{3\deg G -6 }/ J_G \cap\Bbbk[\ud x]_{3\deg G-6} \simeq \bC.
$$


\subsection{Main computations}
Fix a nonzero section of $\Omega^{2}_{\BP^2}(3)\simeq \cO_{\BP^2}$, say,
$$
\Omega :=x_0 dx_1 \wedge dx_2 - x_1 dx_0 \wedge dx_2 + x_2 dx_0 \wedge dx_1.
$$
Given a vector field $Z$ on $\Bbbk^3$, let $i(Z)$ denote the operation of contraction, and 
let $K_j = i(\frac{\partial}{\partial x_j})$. Given a multi-index $J = (j_0,...,j_q)$ of size $q$, let
$$
K_J = K_{j_q} \cdots K_{j_0}, \quad  \Omega_J = K_J \Omega, \quad G_J=G_{j_0} \cdots G_{j_q}.
$$ 

Let $w_0=[W_0],w_2=[W_2]$ correspond to anti-holomorphic cohomology classes and $w_1=[W_1]$ correspond to a holomorphic cohomology class.\footnote{There are $2^3=8$ choices for $\{w_0,w_1,w_2\}$ depending on $w_i$ holomorphic or anti-holomorphic. See Table \ref{tp} for an explanation why we choose this way.} 
Then we can express the Cech-deRham cocycle $W_i$ with respect to the open covering $\cU$ as
\begin{eqnarray} \label{wone}
W_0 &= &\{\frac{U_0}{G_{01}}{\O_{01}}, \frac{U_0}{G_{02}}{\O_{02}},\frac{U_0}{G_{12}}{\O_{01}}  \} \in \cC^0_1,\\ \nonumber
W_1 &=& \{\frac{U_1}{G_0}{\O_{0}}, \frac{U_1}{G_1}{\O_{1}},\frac{U_1}{G_2}{\O_{2}}  \}\in \cC^1_0,\\ \nonumber
W_2 &=& \{\frac{U_2}{G_{01}}{\O_{01}}, \frac{U_2}{G_{02}}{\O_{02}},\frac{U_2}{G_{12}}{\O_{01}}  \} \in  \cC^0_1, \nonumber
\end{eqnarray}

for some polynomials $U_i \in \Bbbk[\ud x]$ with $\deg U_0=\deg U_2 = 2\deg G -3$ and $\deg U_1 = \deg G -3$. Then

$$ 
W_0 \bullet W_1 =(-1)^1 \{ \frac{U_0U_1}{G_{0}G_1^2}{\O_{01}\O_1}, 
\frac{U_0U_1}{G_{0}G_2^2}{\O_{02}\O_2},
\frac{U_0U_1}{G_{1}G_2^2}{\O_{01}\O_2}\} \in \cC_1^1.
$$
$$ 
W_1 \bullet W_2 = \{ \frac{U_1U_2}{G_{0}G_1^2}{\O_{01}\O_1}, 
\frac{U_1U_2}{G_{0}G_2^2}{\O_{02}\O_2},
\frac{U_1U_2}{G_{1}G_2^2}{\O_{01}\O_2}\} \in \cC_1^1.
$$
%
The following theorem is one of the main technical results in \cite{CG}.
\begin{thm} \label{tcg}
The cohomology class $w_i \cup w_j=[W_i \bullet W_j]$ of the cocycle $W_i \bullet W_j$ is trivial 
if and only if $U_i U_j $ belongs to the Jacobian ideal $J_G$ of $G(\ud x)$.
\end{thm}

By the above theorem, if $U_0 U_1 = \sum_{i=0}^2 R_i^{(01)} G_i$ and $U_1 U_2 = \sum_{i=0}^2 R_i^{(12)} G_i$ for $R_i^{(01)}, R_i^{(12)} \in \Bbbk[\ud x]$, then
$W_0 \bullet W_1$ and $W_1 \bullet W_2$ are coboundaries. 
Our first observation is that there is a simple cochain level formula for $W_{01}$ such that $D(W_{01})=W_1 \bullet W_2$: A simple Cech type computation confirms the following proposition.
\begin{prop}\label{wtwo}
If we define
\begin{eqnarray*}
W_{01}:=\{\frac{x_1R_2^{(01)}- x_2R_1^{(01)}}{G_{0}^2}{\O_{0}}, \frac{-x_0 R_2^{(01)} +x_2 R_0^{(01)}}{G_{1}^2}{\O_{1}},\frac{x_0 R_1^{(01)} - x_1 R_0^{(01)}}{G_{2}^2}{\O_{2}}  \} \in \cC_0^1,
\end{eqnarray*}
then $D (W_{01}) = W_1 \bullet W_2$ in $\cC_1^1$. Similarly, if we define
$$
W_{12}=\{-\frac{x_1R_2^{(12)}- x_2R_1^{(12)}}{G_{0}^2}{\O_{0}}, -\frac{-x_0 R_2^{(12)} +x_2 R_0^{(12)}}{G_{1}^2}{\O_{1}},-\frac{x_0 R_1^{(12)} - x_1 R_0^{(12)}}{G_{2}^2}{\O_{2}}  \} \in \cC_0^1,
$$
then $D (W_{12}) = W_2 \bullet W_3$ in $\cC_1^1$.
\end{prop}

\begin{proof}
Because $\O_X^2$ is the sheaf of holomorphic differential 2-forms on $X$ and $\dim_{\bC} X=1$, 
$\cC_0^2=0$. Thus $d(W_{01}) = 0 \in \cC_0^2$ and it is enough to check that $\d(W_{01}) = W_0 \bullet W_1$.
Then this reduces to proving the following identity:
\begin{eqnarray} \label{Cech}
-\frac{U_0U_1}{G_{0}G_1^2}{\O_{01}\O_1}&=&\frac{x_1R_2^{(01)}- x_2R_1^{(01)}}{G_{0}^2}{\O_{0}} - \frac{-x_0 R_2^{(01)} +x_2 R_0^{(01)}}{G_{1}^2}{\O_{1}},
\end{eqnarray}
since the desired identities for other indices will follow from the same computation by the symmetry of indices.
Note that
$$
\O_0=x_2 dx_1-x_1 dx_2, \quad \O_1= x_0dx_2-x_2dx_0, \quad \O_{01}=x_2.
$$
We will use the following identities to prove (\ref{Cech}), which follows from the defining equation $G(\ud x) =0$ of $X_G$:
$$
x_0G_0+x_1G_1 +x_2G_2=\deg(G) \cdot G(\ud x)=0, \quad G_0 dx_0 + G_ 1 dx_1 + G_2 dx_2=0.
$$

\begin{eqnarray*}
RHS &=&\frac{x_1R_2^{(01)}- x_2R_1^{(01)}}{G_{0}^2}{\O_{0}} - \frac{-x_0 R_2^{(01)} +x_2 R_0^{(01)}}{G_{1}^2}{\O_{1}}
 \\
&=& \frac{-(-x_0 R_2^{(01)} +x_2 R_0^{(01)})G_0^2 \O_1+ (x_1R_2^{(01)}- x_2R_1^{(01)})G_1^2\O_0}{G_0^2G_{1}^2}\\
&=& \frac{-(-x_0 R_2^{(01)} +x_2 R_0^{(01)})G_0^2 (x_0dx_2-x_2dx_0)+ (x_1R_2^{(01)}- x_2R_1^{(01)})G_1^2(x_2 dx_1-x_1 dx_2)}{G_0^2G_{1}^2}\\
\end{eqnarray*}

%

\begin{eqnarray*}
LHS &=&-\frac{U_1U_2G_0 \O_{01}\O_1}{G_{0}^2G_1^2}
=-\frac{(R_0^{(01)}G_0 + R_1^{(01)}G_1 +R_2^{(01)}G_2)G_0 \O_{01}\O_1}{G_{0}^2G_1^2}\\
&=& -\frac{(R_0^{(01)}G_0 + R_1^{(01)}G_1 +R_2^{(01)}G_2)G_0 x_2(x_0dx_2-x_2dx_0)}{G_{0}^2G_1^2} \\
&=& -\frac{(R_0^{(01)}G_0 + R_1^{(01)}G_1 +R_2^{(01)}G_2) x_2(x_0G_0)dx_2
-(R_0^{(01)}G_0 + R_1^{(01)}G_1 +R_2^{(01)}G_2)x_2^2(G_0dx_0)}{G_{0}^2G_1^2} 
\end{eqnarray*}


\begin{eqnarray*}
& &(G_0^2 G_1^2) (LHS-RHS) \\
&=&R_2^{(01)} (x_1x_2G_1G_2 +x_1^2G_1^2 -x_0 G_0x_0G_0)dx_2  - R_2^{(01)}( x_2^2G_1G_2+x_1x_2G_1^2)dx_1+R_2^{(01)}x_0x_2G_0G_0 dx_0     \\
&=&R_2^{(01)} (x_1x_2G_1G_2 +x_1^2G_1^2 -x_0 G_0(-x_1G_1-x_2G_2))dx_2 \\
&& - R_2^{(01)}( x_2^2G_1G_2+x_1x_2G_1^2)dx_1+R_2^{(01)}x_0x_2G_0(-G_1dx_1-G_2dx_2)     \\
&=& R_2^{(01)} x_1 G_1 (x_2G_2 + x_1G_1 +x_0G_0) dx_2  - R_2^{(01)} x_2 G_1 (x_2G_2 + x_1G_1 +x_0G_0) dx_1=0.
\end{eqnarray*}

\end{proof}

The $\Bbbk$-valued Massey triple product $\langle w_0, w_1, w_2 \rangle$ of $w_0,w_1,$ and $w_2$ is, by definition, the image of the cohomology class $[W_{01}\bullet W_2 + W_{0}\bullet W_{12}]
\in H^2(X_G) \simeq H^2 ( \oplus_{p+q=\bullet}\cC_q^p, D)$ under an isomorphism $\theta: H^2(X_G) \buildrel \simeq \over \to \Bbbk$.

Let $\eta: X_G \to \BP^2$ be a given closed embedding. In order to compute $\langle w_0, w_1, w_2 \rangle$, we use the connecting coboundary map induced from the Poincare residue exact sequence of sheaves on $\BP^2$:
\begin{eqnarray}\label{pres}
0 \to \Omega^\bullet_{\BP^2} \to \Omega^\bullet_{\BP^2}(\log X_G) \mapto{\res} \eta_* \Omega_{X_G}^{\bullet-1} \to 0,
\end{eqnarray}
where $\Omega^p_{\BP^2}$ is a sheaf of regular $p$-forms on $\BP^2$ and $\Omega^p_{\BP^2}(\log X_G)$
is a sheaf of rational $p$-forms $\omega$ on $\BP^2$ such that $\o$ and $d \o$ are regular on $\BP^2\setminus X$ and have at most a pole of order one along $X$. See the page 444 in \cite{PS} for more details.
In \cite{CG}, the authors gave an explicit formula of the coboundary map $\tilde\d$ in the Poincar\'e residue sequence induced from \ref{pres}
$$
\tilde\d: H^1 (\cU |_{X_G}, \O_{X_G}^1) \to H^2 (\BP^2, \O_{\BP^2}^2).
$$

For the computation of $\tilde \d$, we need to compute the lift of
$W_{01}\bullet W_2 + W_{0}\bullet W_{12}$, namely, 
$$
(W_{01}\bullet W_2 + W_{0}\bullet W_{12})\wedge \frac{dG}{G}
\in \cC^1(\cU, \O^2_{\BP^2} (\log X_G))
,
$$
and then apply the Cech coboundary map $\d$:
$$
\left[\d \left((W_{01}\bullet W_2 + W_{0}\bullet W_{12})\wedge \frac{dG}{G}  \right)\right] = \tilde \d( [W_{01}\bullet W_2 + W_{0}\bullet W_{12}])
\quad \text{ inside } H^2 (\BP^2, \O_{\BP^2}^2).
$$

Following \cite{CG}, consider the complexes $\cC_\ell^\bullet(\cU, \O_{\BP^2}^2)$ where a typical cochain $\phi$ has the form
$
\frac{R_{012}\Omega}{G_{012}^\ell}
$
where $R_{012} \in \Bbbk[\ud x]$ with $\deg (R_{012})=3\ell(\deg G -1)-3$.
If one denotes the cohomology groups of this complex by $H_\ell^\bullet(\BP^2, \O_{\BP^2}^2)$, then one has
$$
H^\bullet(\BP^2, 
\O_{\BP^2}^2) = \mathop{\lim_{\longrightarrow}} H^\bullet_\ell (\BP^2, 
\O_{\BP^2}^2),
$$
where the map from $\cC_\ell^\bullet(\cU, \O_{\BP^2}^2)$ to $\cC_{\ell+1}^\bullet(\cU, \O_{\BP^2}^2)$ is given by multiplication of the
numerator in the cocyle against $G_{012}$.
According to the lemma (\cite{CG}, p18), the Grothedieck residue map (see \ref{GR} our explicit 
normalization of $\res_0$)
$$
\res_0: H^2_\ell(\BP^2, \O_{\BP^2}^2) \longrightarrow \Bbbk
$$
is an isomorphism which commutes with multiplication by $G_{012}$.
If we let $\theta = \res_0 \circ \tilde \d$, then
$$
\langle w_0, w_1, w_2 \rangle = \res_0 \left( \tilde \d( [W_{01}\bullet W_2 + W_{0}\bullet W_{12}])  \right).
$$

For simplicity, we use the following notation:
\begin{eqnarray} \label{nt}
A^{(ij)}_0:=x_1 R_2^{(ij)} - x_2 R_1^{(ij)}, \quad A^{(ij)}_1:=-x_0 R_2^{(ij)} + x_2 R_0^{(ij)},
\quad A^{(ij)}_2:= x_0 R_1^{(ij)} - x_1 R_0^{(12)},
\end{eqnarray}
for $(ij)=(01)$ or $(12)$. One can easily check that
$$
\frac{\O_{ij}\O_i}{G_i}= \frac{\O_{ij}\O_j}{G_j}.
$$
for $i, j=0,1,2$ with $i < j$. Then the above proposition \ref{wtwo} says that
\begin{eqnarray}\label{keycomp}
-\frac{U_kU_\ell \O_{ij}\O_j}{G_{ij} G_j}= -\frac{U_kU_\ell \O_{ij}\O_i}{G_{ij} G_i}=
(-1)^i\frac{A_i^{(k\ell)}\O_i}{G_i^2} +(-1)^j\frac{A_j^{(k\ell)}\O_j}{G_j^2}
\end{eqnarray}
for $(k, \ell) =(0,1)$ or $(1,2)$ and $i, j=0,1,2$ with $i < j$.

\begin{lem} \label{res}
We have the following formula
\begin{eqnarray*}
\d \left((W_{01}\bullet W_2 + W_{0}\bullet W_{12})\wedge \frac{dG}{G}\right) =\deg G\cdot \frac{ A -B}{G_{012}^2}\cdot \Omega \in \cC_2^2(\cU, \O_{\BP^2}^2),
\end{eqnarray*}
where
\begin{eqnarray*}
A&=&U_0 x_1 R_0^{(12)} G_0G_1+U_2x_0R_2^{(01)}G_0G_2 +
U_0x_0 R_0^{(12)}G_0G_0, \\
B&=&A_0^{(01)} U_2 G_1G_2 
+ U_0 A_1^{(12)}G_0G_2
+ U_0U_1U_2 x_0G_0.
\end{eqnarray*}
\end{lem}

\begin{proof}

Note that 
$$
W_{01}\bullet W_2 =W_{01} \wedge W_2= \{ 
\frac{A_0^{(01)}U_2}{G_{0}^2G_{12}}{\O_{0}\O_{01}},
\frac{A_0^{(01)}U_2}{G_{0}^2G_{02}}{\O_{0}\O_{02}},
 \frac{A_1^{(01)}U_2}{G_{1}^2G_{12}}{\O_{1}\O_{01}} 
  \}\in \cC_1^1,
$$
and
$$
W_{0}\bullet W_{12} =-W_0 \wedge W_{12}= \{ 
\frac{U_0A_1^{(12)}}{G_{1}^2G_{12}}{\O_{1}\O_{01}},
\frac{U_0A_2^{(12)}}{G_{2}^2G_{02}}{\O_{2}\O_{02}},
 \frac{U_0A_2^{(12)}}{G_{2}^2G_{12}}{\O_{2}\O_{01}} 
  \}\in \cC_1^1.
$$
By using the lemma (\cite{CG}, p 14), we get the following formulae:
$$
(W_{01}\bullet W_2 + W_{0}\bullet W_{12})\wedge \frac{dG}{G}
=\{ S_{01}  \frac{\O}{G},  S_{02}  \frac{\O}{G}, S_{01}\frac{\O}{G}  \}
\in \cC^1(\cU, \O^2_{\BP^2} (\log X_G)),
$$
where 
\begin{eqnarray*}
S_{01} &=& \frac{-A_0^{(01)}U_2 x_2}{G_0G_{12}} + \frac{-U_0A_1^{(12)} x_2}{G_{12}G_1},\\
S_{02} &=& \frac{A_0^{(01)}U_2 x_1}{G_0G_{02}} + \frac{-U_0A_2^{(12)} x_1}{G_{02}G_2},\\
S_{12} &=& \frac{-A_1^{(01)}U_2 x_0}{G_1G_{12}} + \frac{-U_0A_2^{(12)} x_0}{G_{12}G_2}.
\end{eqnarray*}
Therefore $\d \left((W_{01}\bullet W_2 + W_{0}\bullet W_{12})\wedge \frac{dG}{G}\right)$ is equal to
{\small
\begin{eqnarray*}
-\left[\left(\frac{A_0^{(01)}U_2G_1 + U_0A_1^{(12)} G_0}{G_{12}G_{012}} \right) x_2G_2 + 
\left(\frac{A_0^{(01)}U_2G_2 + U_0A_2^{(12)} G_0}{G_{02}G_{012}} \right) x_1G_1  +
\left(\frac{A_1^{(01)}U_2G_2 + U_0A_2^{(12)} G_1}{G_{12}G_{012}} \right) x_0G_0\right] \frac{\O}{G}.
\end{eqnarray*}
}

Using the relation $-x_2 G_2=x_0 G_0 +x_1 G_1-\deg G \cdot G$ in $\BP^2$, 
we continue the computation and 
$$\d \left((W_{01}\bullet W_2 + W_{0}\bullet W_{12})\wedge \frac{dG}{G}\right)$$ 
is equal to
\begin{eqnarray*}
& &\left(\frac{A_0^{(01)}U_2G_1 + U_0A_1^{(12)} G_0}{G_{12}G_{012}} \right) (x_0 G_0 +x_1 G_1-\deg G \cdot G)\frac{\O}{G}   \\
& &-\left(\frac{A_0^{(01)}U_2G_2 + U_0A_2^{(12)} G_0}{G_{02}G_{012}} \right) x_1G_1  
-\left(\frac{A_1^{(01)}U_2G_2 + U_0A_2^{(12)} G_1}{G_{12}G_{012}} \right) x_0G_0 \frac{\O}{G}\\
& =&-\deg G \left( \frac{A_0^{(01)}U_2G_1 + U_0 A_1^{(12)} G_0}{G_{12}G_{012}}\right) \O + \left(\frac{U_0A_1^{(12)} G_0G_2-U_0A_2^{(12)}G_1G_0}{G_{012}^2}  \right)x_1G_1\frac{\O}{G} \\
 & &+ \left(\frac{A_0^{(01)}U_2 G_1G_2-A_1^{(01)}U_2G_2G_0 +
 U_0A_1^{(12)}G_0G_2-U_0A_2^{(12)}G_1G_0}{G_{012}^2} \right)x_0G_0 \frac{\O}{G}.
\end{eqnarray*}

On the other hand, we have that
\begin{eqnarray*}
 &&U_0A_1^{(12)}G_0G_2-U_0A_2^{(12)}G_1G_0\\
 &=& U_0G_0 (A_1^{(12)}G_2 - A_2^{(12)} G_1)\\
 &=& U_0G_0 \left( (-x_0 R_2^{(12)} + x_2 R_0^{(12)})G_2  - (x_0 R_1^{(12)}-x_1 R_0^{(12)}) G_1 \right) \\
  &=& U_0G_0 \left( (x_2G_2 + x_1G_1)R_0^{(12)}  - x_0 (R_2^{(12)}G_2 + R_1^{(12)})G_1)\right) \\
   &=& U_0G_0 \left( \deg G \cdot G \cdot R_0^{(12)} - x_0 (R_0^{(12)}G_0+R_1^{(12)}G_1+R_2^{(12)}G_2)\right) \\
   &=& \deg G \cdot U_0 R_0^{(12)} G_0 G - U_0U_1U_2 x_0 G_0.
\end{eqnarray*}

Similarly, we have that
$$
A_0^{(01)}U_2 G_1G_2-A_1^{(01)}U_2G_2G_0 = 
\deg G \cdot U_2 R_2^{(01)} G_2 G - U_0U_1U_2 x_2 G_2.
$$

By substituting these terms in the equation regarding $\d \left((W_{01}\bullet W_2 + W_{0}\bullet W_{12})\wedge \frac{dG}{G}\right)$, a straightforward computation confirms that the numerator involved is divisible by $G$ and this proves the lemma.
\end{proof}

\section{An explicit algorithm for the Massey triple product}

\subsection{Main theorem}
Let $g$ be the genus of $X$ so that $\dim_\bC H^0(X,\O_X^1)=\dim_\bC H^1(X,\O_X^0)=g$.
We choose all the possible $U_1 \in \Bbbk[\ud x]_{\deg G -3 }$ (modulo $J_G$) and  all the possible $U_0, U_2  \in \Bbbk[\ud x]_{2\deg G -3 }$  (modulo $J_G$) such that 
(this can be done by using a computer algebra program)
\begin{eqnarray} \label{con}
U_0 U_1 = \sum_{i=0}^2 R_i^{(01)} G_i, \quad U_1 U_2 = \sum_{i=0}^2 R_i^{(12)} G_i,
\end{eqnarray}
for $R_i^{(01)}, R_i^{(12)} \in \Bbbk[\ud x]$.
Then $W_0 \bullet W_1$ and $W_1 \bullet W_2$ are coboundaries by Theorem \ref{tcg}.

The Massey triple product is not defined for all the elements of the first singular cohomology group $H^1(X, \bC)$. Choosing $U_0, U_1, U_2$ satisfying \ref{con} amounts to finding a maximal isotropic subspace of $H^1(X, \bC)$.
Define 
$$
A_X:=\{W_0, W_1, W_2, W_{01}, W_{12}\in \cC_0^1 \oplus \cC_1^0   : U_0, U_1, U_2 \text{ satisfies } \ref{con}\},
$$
where $W_i$ are given in \ref{wone} and $W_{ij}$ are given in Proposition \ref{wtwo}.
Then Proposition \ref{wtwo} implies that 
$A_X$ is a defining system (Definition \ref{mase}) for a Massey triple product for $w_i= [W_i]$, $i=0,1,2$.
The key point of Proposition \ref{wtwo} is that it gives an explicitly computable expression for the Cech representative $W_{ij}$ in terms of polynomials in $\Bbbk[\ud x]$.

We compute the Massey triple product of $w_0, w_1, w_2$ with respect to $A_X$ by the formula:
$$
\langle w_0, w_1, w_2 \rangle_{A_X} := \res_0 \left( \tilde \d( [W_{01}\bullet W_2 + W_{0}\bullet W_{12}])  \right),
$$
using
$$
\theta: H^1(\cU|_X, \O_{X}^1)  \xrightarrow{\tilde \delta} H^2_2(\cU|_{\BP^2}, \O_{\BP^2}^2)
 \xrightarrow{\res_0} \bC.
$$
Lemma \ref{res} gives an explicitly computable expression for $\tilde \delta$:
\begin{eqnarray*}
\tilde\delta \left([W_{01}\bullet W_2 + W_{0}\bullet W_{12}]\right) 
=[\deg G\cdot \frac{ A -B}{G_{012}^2}\cdot \Omega] \in H_2^2(\cU|_{\BP^2}, \O_{\BP^2}^2),
\end{eqnarray*}
where $A, B$ are given explicitly in Lemma \ref{res}.

The final thing is to compute $\res_0: H^2_2(\BP^2, \O_{\BP^2}^2) \simeq \Bbbk$ explicitly.
For this, let $J$ be the ideal of $\Bbbk[\ud x]$ generated by $G_0(\ud x)^2, G_1(\ud x)^2, G_2(\ud x)^2.$
Then there is an isomorphism (see Corollary (iv), p 18, \cite{CG})
$$
\left(\Bbbk[\ud x]/J\right)_{6 \deg G -9} \simeq \Bbbk.
$$
Thus, if we put
$$
\mathrm{Det}_G(\ud x):= \det \left( \frac{\partial G_i(\ud x)^2 }{\partial x_j}\right) \in \Bbbk[\ud x]_{6 \deg G -9},
$$
then $\mathrm{Det}_G(\ud x)$ does not belong to $J$ and, for any given $F(\ud x) \in\Bbbk[\ud x]_{6 \deg G -9}$, there is a unique 
$R_F \in \Bbbk$ such that
$$
F(\ud x) =  \sum_{i=0}^2 v_i(\ud x)G_i(\ud x)^2  + R_F \cdot \mathrm{Det}_G(\ud x), 
\quad v_i(\ud x) \in \Bbbk[\ud x].
$$
The fact that $\res_0$ factors through $F\mapsto R_F$ (see Lemma and its proof, p 18, \cite{CG})
$$
\res_0: H^2_2(\BP^2, \O_{\BP^2}^2) 
\simeq \left(\Bbbk[\ud x]/J\right)_{6 \deg G -9} \simeq \Bbbk
$$
implies that, Then 
\begin{eqnarray}\label{GR}
\res_0 \left( \frac{F(\ud x)}{G_{012}^2} \cdot \Omega\right) = R_F.
\end{eqnarray}

So our final formula is the following:
\begin{thm} \label{mt}
Let $w_i=[W_i]$ correspond to $U_i$ satisfying \ref{con}. 
Assuming all the notations so far, we have that
$$
\langle w_0, w_1, w_2 \rangle_{A_X} =
R_{\deg G \cdot (A-B)},
$$
where
\begin{eqnarray*}
A&=&U_0 x_1 R_0^{(12)} G_0G_1+U_2x_0R_2^{(01)}G_0G_2 +
U_0x_0 R_0^{(12)}G_0G_0, \\
B&=&A_0^{(01)} U_2 G_1G_2 + U_0 A_1^{(12)}G_0G_2
+ U_0U_1U_2 x_0G_0.
\end{eqnarray*}
\end{thm}
See \ref{nt} and \ref{con} for the notations $R_i^{(jk)}$ and $A_i^{(jk)}$.
This theorem implies that the Massey triple product are completely computable using
a computer algebra program.
In particular, we have
\begin{cor}
The Massey triple product $\langle w_0, w_1, w_2 \rangle_{A_X}$ relative to $A_X$ vanishes if and only if $A-B \in J=\langle G_0(\ud x)^2, G_1(\ud x)^2, G_2(\ud x)^2 \rangle$.
\end{cor}


\begin{cor} \label{indep}
Vanishing of $\langle w_0,w_1,w_2\rangle_{A_X}$ does not depend on choices of cohomology representatives of $[W_0],[W_1],[W_2]$.
\end{cor}

\begin{proof}
Since a cohomology representative $W_i$ of $w_i$ corresponds to $U_i$ via the Griffiths isomorphism \ref{Gisom}, the statement that $W_i$ is cohomologous to $\widetilde{W}_i$ is equivalent to that $U_i-\widetilde{U}_i$ belongs to the Jacobian ideal $J_G$ (we use the notation $\ \widetilde{\cdot} \ $ for quantities corresponding to $\widetilde{W}_i$). 
It suffices to prove the corollary for each $U_i, \ i=0,1,2,$ one by one.
We start from $U_0$. Let
$\widetilde{U}_0:=U_0+T_0G_0+T_1G_1+T_2G_2$ for $T_i \in \bC[\ud x]$.
In this case,
\begin{align*}
&\widetilde{R}^{(01)}_i=R^{(01)}_i+T_iU_1,\quad\widetilde{R}^{(12)}_i=R^{(12)}_i, \quad i =0,1,2,\\
&\widetilde{A}^{(01)}_k=A^{(01)}_k+T_{k+2}x_{k+1}U_1-T_{k+1}x_{k+2}U_1,\quad\widetilde{A}^{(12)}_k=A^{(12)}_k, \quad k \equiv 0,1,2 \mod 3.
\end{align*}
\begin{align*}
\widetilde{A}&\equiv\widetilde{U}_0R^{(12)}_0G_0G_1+U_2x_0\widetilde{R}^{(01)}_2G_0G_2\mod J\\
&\equiv(U_0+T_2G_2)x_1R^{(12)}_0G_0G_1+U_2x_0\left(R^{(01)}_2+T_2U_1\right)G_0G_2\mod J\\
&\equiv A+T_2x_1R^{(12)}_0G_0G_1G_2+T_2x_0R^{(12)}_1G_0G_1G_2\mod J.
\end{align*}
\begin{align*}
\widetilde{B}&\equiv\widetilde{A}^{(01)}_0U_2G_1G_2+\widetilde{U}_0A^{(12)}_1G_0G_2+\widetilde{U}_0U_1U_2x_0G_0\mod J\\
&\equiv\left(A^{(01)}_0+T_2x_1U_1-T_1x_2U_1\right)U_2G_1G_2+(U_0+T_1G_1)A^{(12)}_1G_0G_2\\
&\quad\quad+(U_0+T_1G_1+T_2G_2)U_1U_2x_0G_0\mod J\\
&\equiv B+(T_2x_1-T_1x_2)R^{(12)}_0G_0G_1G_2+T_1A^{(12)}_1G_0G_1G_2\\
&\quad\quad+T_1x_0R^{(12)}_2G_0G_1G_2+T_2x_0R^{(12)}_1G_0G_1G_2\mod J\\
&\equiv B+T_2x_1R^{(12)}_0G_0G_1G_2+T_2x_0R^{(12)}_1G_0G_1G_2\mod J
\end{align*}
Therefore, $\widetilde{A}-\widetilde{B}\equiv A-B\mod J$.
If we let $\widetilde{U}_2:=U_2+T_0G_0+T_1G_1+T_2G_2$ for $T_i\in \bC[\ud x]$,
then a similar computation again confirms that $\widetilde{A}-\widetilde{B}\equiv A-B\mod J$.
Finally, let $\widetilde{U}_1:=U_1+T_0G_0+T_1G_1+T_2G_2$.
Since $\deg\widetilde{U}_1=\deg U_1=\deg G-3$, the only possibility is
\begin{align*}
T_0G_0+T_1G_1+T_2G_2=0
\end{align*}
Therefore, $\widetilde{R}^{(ij)}_k=R^{(ij)}_k$ and $\widetilde{A}^{(ij)}_k=A^{(ij)}_k$. Hence $\widetilde{A}-\widetilde{B}\equiv A-B \mod J$ in this case, too.
\end{proof}

\begin{cor}
	If $\deg G=3$, then $A-B\in J$, i.e., $\langle w_0,w_1,w_2\rangle_{A_X}=0$ for all $w_0,w_1,w_2$.
\end{cor}

\begin{proof}
	Since $\deg G=3$, we can assume that $U_1=1$, $U_0=\sum_{i=0}^2R_i^{(01)}G_i$, and $U_2=\sum_{i=0}^2R_i^{(12)}G_i$. By Theorem \ref{mt},
	\begin{eqnarray*}
		A&=&U_0x_1R_0^{(12)}G_0G_1+U_2x_0R_2^{(01)}G_0G_2+U_0x_0R_0^{(12)}G_0G_0,\\
		&=&\Big(\sum_{i=0}^2R_i^{(01)}G_i\Big)x_1R_0^{(12)}G_0G_1+\Big(\sum_{i=0}^2R_i^{(12)}G_i\Big)x_0R_2^{(01)}G_0G_2+\Big(\sum_{i=0}^2R_i^{(01)}G_i\Big)x_0R_0^{(12)}G_0G_0,\\
		&\equiv&x_1R_2^{(01)}R_0^{(12)}G_0G_1G_2+x_0R_2^{(01)}R_1^{(12)}G_0G_1G_2\mod J,\\
		B&=& A_0^{(01)}U_2G_1G_2+U_0A_1^{(12)}G_0G_2+U_0U_1U_2x_0G_0,\\
		&=& (x_1R_2^{(01)}-x_2R_1^{(01)})\Big(\sum_{i=1}^2R_i^{(12)}G_i\Big)G_1G_2+\Big(\sum_{i=1}^2R_i^{(01)}G_i\Big)(-x_0R_2^{(12)}+x_2R_0^{(12)})G_0G_2\\
		&&+\Big(\sum_{i=1}^2R_i^{(01)}G_i\Big)\Big(\sum_{i=1}^2R_i^{(12)}G_i\Big)x_0G_0,\\
		&\equiv& (x_1R_2^{(01)}-x_2R_1^{(01)})R_0^{(12)}G_0G_1G_2+R_1^{(01)}G_1(-x_0R_2^{(12)}+x_2R_0^{(12)})G_0G_2\\
		&&+R_1^{(01)}G_1R_2^{(12)}G_2x_0G_0+R_2^{(01)}G_2R_1^{(12)}G_1x_0G_0\mod J,\\
		&\equiv&x_1R_2^{(01)}R_0^{(12)}G_0G_1G_2+x_0R_2^{(01)}R_1^{(12)}G_0G_1G_2\mod J.
	\end{eqnarray*}
	Therefore, $A-B\equiv0\mod J$, i.e., $A-B\in J$.
\end{proof}

\subsection{Other possible Massey triple products}
We explain why we chose $w_0,w_2$ anti-holomorphic and $w_1$ holomorphic in order to get a non-trivial defining system. 
The first line in Table \ref{tp} is the case we considered in Theorem \ref{mt}.
The second and third lines in Table \ref{tp} can be considered in the same way as Theorem \ref{mt}. For other remaining choices, we do not know how to get some non-zero $W_{ij}$ such that $D(W_{ij})=W_i \cup W_j$ and the Massey triple product is zero for trivial reasons (the cup product of two holomorphic forms (respectively, two anti-holomorphic forms) is zero). 
Note that we do not know an explicit expression\footnote{We only know $W_{ij} \in \cC_0^1$ such that $D(W_{ij})=W_i \bullet W_j$ in Proposition \ref{wtwo}.} $\tilde W_{ij} \in \cC_1^0$ such that $D(\tilde W_{ij})=W_i \bullet W_j \in \cC_1^1$ when $w_i \cup w_j =0$. The asterisk * in Table \ref{tp} means that it is not zero in general.

\begin{table}
\caption{Defining systems for the Massey triple product: $A=\{W_0,W_1,W_2,W_{01},W_{12} \}$}\label{tp}
\begin{tabular}{|c | c | c | c | c |c | } 
\hline
 $W_0$ & $W_1$   & $W_2$      &  $W_{01}$     &  $W_{12}$      & $\langle W_0,W_1,W_2\rangle_A$     \\ \hline 
  anti-holomorphic     & holomorphic      & anti-holomorphic &  *  &    *   &    *  \\ \hline 
  anti-holomorphic
 & anti-holomorphic     & holomorphic      & 0      & *      & *     \\ \hline 
   holomorphic
 & anti-holomorphic& anti-holomorphic & *  & 0 &  * \\
 \hline
 holomorphic
 & anti-holomorphic      & holomorphic   &  *    &  *   & 0    \\ \hline 
anti-holomorphic
 & anti-holomorphic    & anti-holomorphic     & 0      & 0      & 0     \\ \hline 
 holomorphic
 & holomorphic   & holomorphic     & 0      & 0      & 0     \\ \hline 
 holomorphic
 & holomorphic     &   anti-holomorphic    & 0     & * & 0     \\ \hline 
 anti-holomorphic
 & holomorphic      & holomorphic      & *      & 0     & 0     \\ \hline 
\end{tabular}
\end{table}

\section{Computer algorithm and experiment}
The goal of this section is to provide an explicit computer algorithm and to calculate triple Massey products of some examples. 
Our example is a family of smooth projective curve $X_G$ given by $G(\underline{x}) = x_0^n+x_1^n+x_2^n$ as $n$ varies. Let us use the following simplified notations:
\[
  \langle U_0,U_1,U_2\rangle=\langle w_0,w_1,w_2\rangle_{A_X},
\]
where the relations between $U_i$ and $w_i$ are give in (\ref{wone}). We use the program \textsf{SageMath8.3} for running the algorithms.\par
First, we have to find $U_0,U_1,U_2$ for computing the triple Massey product such that $U_0U_1,U_1U_2\in J_G$. Our method for choosing $U_0,U_1$ such that $U_0U_1\in J_G$ is to generate repeatedly homogeneous polynomials $U_0,U_1$ randomly until the multiplication $U_0U_1$ is in $J_G$.
During this repeated process, we analyze the relation between the number of monomial terms of selected polynomials $U_0, U_1$ and vanishing behavior of the cup product. For this, we choose 50,000 pairs of homogeneous polynomials $U_0$ and $U_1$ such that $\mathrm{deg}U_0=2n-3$ and $\mathrm{deg}U_2=n-3$ and then calculate the ratio of pairs $U_0,U_1$ among 50,000 pairs such that $U_0U_1\in J_G$. 
Define 
$$
M_1(n) = \frac{(n-1)(n-2)}{2}, \quad M_2(n)=\frac{(2n-1)(2n-2)}{2}
$$
 where $M_1(n)$ (respectively, $M_2(n)$) is the number of monomials of degree $n-3$ (respectively, $2n-3$).
Then we calculate the ratio of vanishing cup product 
when the number of monomial terms of randomly chosen $U_0$ (respectively, $U_1$) is less than or equal to $\ulcorner\frac{M_2(n)}{\ell}\urcorner$ (respectively, $\ulcorner\frac{M_1(n)}{\ell}\urcorner$) as we vary $\ell =1,2, \cdots$.
The results are recorded in the graphs of Figure \ref{fig:grf}.
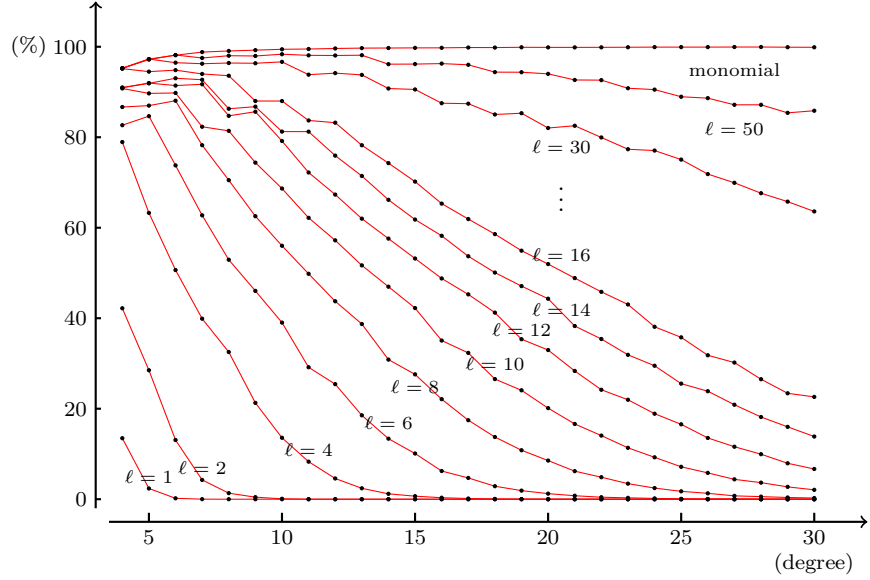
\begin{figure}[] 
  \begin{center}
    \begin{tikzpicture}[yscale=0.06,xscale=0.35]
      \draw[thick,->] (3.5,-5)--(32,-5);
      \draw[thick,->] (3,-2)--(3,110);
      \draw[thick,-] (3,0)--(3.2,0);
      \draw[thick,-] (3,20)--(3.2,20);
      \draw[thick,-] (3,40)--(3.2,40);
      \draw[thick,-] (3,60)--(3.2,60);
      \draw[thick,-] (3,80)--(3.2,80);
      \draw[thick,-] (3,100)--(3.2,100);
      \draw[thick,-] (5,-5)--(5,-3);
      \draw[thick,-] (10,-5)--(10,-3);
      \draw[thick,-] (15,-5)--(15,-3);
      \draw[thick,-] (20,-5)--(20,-3);
      \draw[thick,-] (25,-5)--(25,-3);
      \draw[thick,-] (30,-5)--(30,-3);
      \node[below] at (5,-5) {\footnotesize5};
      \node[below] at (10,-5) {\footnotesize10};
      \node[below] at (15,-5) {\footnotesize15};
      \node[below] at (20,-5) {\footnotesize20};
      \node[below] at (25,-5) {\footnotesize25};
      \node[below] at (30,-5) {\footnotesize30};
      \node[below] at (30,-10) {\footnotesize(degree)};
      \node[left] at (3,0) {\footnotesize0};
      \node[left] at (3,20) {\footnotesize20};
      \node[left] at (3,40) {\footnotesize40};
      \node[left] at (3,60) {\footnotesize60};
      \node[left] at (3,80) {\footnotesize80};
      \node[left] at (3,100) {\footnotesize100};
      \node[left] at (1.5,100) {\footnotesize(\%)};
      \draw[red,-](4,13.48)--(5,2.37)--(6,0.20)--(7,0.01)--(8,0.00)--(9,0.00)--(10,0.00)--(11,0.00)--(12,0.00)--(13,0.00)--(14,0.00)--(15,0.00)--(16,0.00)--(17,0.00)--(18,0.00)--(19,0.00)--(20,0.00)--(21,0.00)--(22,0.00)--(23,0.00)--(24,0.00)--(25,0.00)--(26,0.00)--(27,0.00)--(28,0.00)--(29,0.00)--(30,0.00);
      \draw[red,-](4,42.20)--(5,28.50)--(6,13.08)--(7,4.26)--(8,1.34)--(9,0.45)--(10,0.09)--(11,0.00)--(12,0.00)--(13,0.00)--(14,0.00)--(15,0.00)--(16,0.00)--(17,0.00)--(18,0.00)--(19,0.00)--(20,0.00)--(21,0.00)--(22,0.00)--(23,0.00)--(24,0.00)--(25,0.00)--(26,0.00)--(27,0.00)--(28,0.00)--(29,0.00)--(30,0.00);
      \draw[red,-](4,78.94)--(5,63.28)--(6,50.65)--(7,39.90)--(8,32.53)--(9,21.29)--(10,13.58)--(11,8.30)--(12,4.58)--(13,2.41)--(14,1.19)--(15,0.67)--(16,0.33)--(17,0.15)--(18,0.06)--(19,0.02)--(20,0.00)--(21,0.00)--(22,0.00)--(23,0.00)--(24,0.00)--(25,0.00)--(26,0.00)--(27,0.00)--(28,0.00)--(29,0.00)--(30,0.00);
      \draw[red,-](4,82.66)--(5,84.66)--(6,73.78)--(7,62.74)--(8,52.93)--(9,46.04)--(10,39.04)--(11,29.16)--(12,25.43)--(13,18.54)--(14,13.38)--(15,10.10)--(16,6.23)--(17,4.70)--(18,2.87)--(19,1.90)--(20,1.23)--(21,0.77)--(22,0.44)--(23,0.30)--(24,0.18)--(25,0.08)--(26,0.05)--(27,0.04)--(28,0.01)--(29,0.00)--(30,0.00);
      \draw[red,-](4,86.68)--(5,86.99)--(6,88.07)--(7,78.24)--(8,70.51)--(9,62.55)--(10,56.02)--(11,49.82)--(12,43.75)--(13,38.72)--(14,30.86)--(15,27.59)--(16,22.11)--(17,17.46)--(18,13.75)--(19,10.82)--(20,8.54)--(21,6.20)--(22,4.86)--(23,3.44)--(24,2.46)--(25,1.74)--(26,1.30)--(27,0.74)--(28,0.57)--(29,0.38)--(30,0.26);
      \draw[red,-](4,90.78)--(5,89.69)--(6,89.79)--(7,82.32)--(8,81.41)--(9,74.39)--(10,68.67)--(11,62.19)--(12,57.21)--(13,51.68)--(14,46.97)--(15,42.25)--(16,35.07)--(17,32.34)--(18,26.55)--(19,24.07)--(20,20.13)--(21,16.63)--(22,14.06)--(23,11.36)--(24,9.26)--(25,7.16)--(26,5.82)--(27,4.39)--(28,3.66)--(29,2.73)--(30,2.07);
      \draw[red,-](4,90.97)--(5,91.98)--(6,91.42)--(7,91.71)--(8,84.73)--(9,85.64)--(10,79.17)--(11,72.21)--(12,67.33)--(13,62.01)--(14,57.60)--(15,53.20)--(16,48.79)--(17,45.27)--(18,41.24)--(19,35.36)--(20,32.97)--(21,28.35)--(22,24.19)--(23,21.98)--(24,18.90)--(25,16.55)--(26,13.54)--(27,11.60)--(28,9.95)--(29,7.94)--(30,6.68);
      \draw[red,-](4,90.96)--(5,91.91)--(6,93.04)--(7,92.71)--(8,86.30)--(9,86.76)--(10,81.25)--(11,81.24)--(12,75.94)--(13,71.44)--(14,66.17)--(15,61.81)--(16,58.22)--(17,53.70)--(18,50.09)--(19,47.10)--(20,44.31)--(21,38.28)--(22,35.41)--(23,31.91)--(24,29.50)--(25,25.54)--(26,23.87)--(27,20.88)--(28,18.17)--(29,15.99)--(30,13.85);
      \draw[red,-](4,95.16)--(5,94.50)--(6,94.83)--(7,93.99)--(8,93.60)--(9,88.01)--(10,88.00)--(11,83.71)--(12,83.22)--(13,78.23)--(14,74.28)--(15,70.20)--(16,65.33)--(17,61.92)--(18,58.59)--(19,54.93)--(20,51.97)--(21,48.87)--(22,45.83)--(23,43.04)--(24,38.11)--(25,35.77)--(26,31.81)--(27,30.20)--(28,26.52)--(29,23.41)--(30,22.61);
      \draw[red,-](4,95.19)--(5,97.28)--(6,96.48)--(7,96.25)--(8,96.41)--(9,96.35)--(10,96.66)--(11,93.83)--(12,94.17)--(13,93.78)--(14,90.77)--(15,90.56)--(16,87.52)--(17,87.42)--(18,85.03)--(19,85.31)--(20,82.04)--(21,82.54)--(22,79.95)--(23,77.35)--(24,77.06)--(25,75.06)--(26,71.86)--(27,69.93)--(28,67.63)--(29,65.77)--(30,63.60);
      \draw[red,-](4,95.28)--(5,97.17)--(6,98.17)--(7,97.52)--(8,98.02)--(9,97.96)--(10,98.35)--(11,98.10)--(12,98.07)--(13,98.13)--(14,96.15)--(15,96.17)--(16,96.28)--(17,96.01)--(18,94.39)--(19,94.37)--(20,94.01)--(21,92.66)--(22,92.62)--(23,90.83)--(24,90.53)--(25,88.95)--(26,88.64)--(27,87.17)--(28,87.17)--(29,85.39)--(30,85.86);
      \draw[red,-](4,95.26)--(5,97.28)--(6,98.13)--(7,98.82)--(8,99.07)--(9,99.25)--(10,99.44)--(11,99.50)--(12,99.62)--(13,99.70)--(14,99.72)--(15,99.75)--(16,99.76)--(17,99.83)--(18,99.82)--(19,99.88)--(20,99.86)--(21,99.86)--(22,99.87)--(23,99.89)--(24,99.92)--(25,99.92)--(26,99.92)--(27,99.94)--(28,99.94)--(29,99.90)--(30,99.88);

      \draw[fill=black] (4,13.48) ellipse (0.06cm and 0.35cm);
      \draw[fill=black] (5,2.37) ellipse (0.06cm and 0.35cm);
      \draw[fill=black] (6,0.20) ellipse (0.06cm and 0.35cm);
      \draw[fill=black] (7,0.01) ellipse (0.06cm and 0.35cm);
      \draw[fill=black] (8,0.00) ellipse (0.06cm and 0.35cm);
      \draw[fill=black] (9,0.00) ellipse (0.06cm and 0.35cm);
      \draw[fill=black] (10,0.00) ellipse (0.06cm and 0.35cm);
      \draw[fill=black] (11,0.00) ellipse (0.06cm and 0.35cm);
      \draw[fill=black] (12,0.00) ellipse (0.06cm and 0.35cm);
      \draw[fill=black] (13,0.00) ellipse (0.06cm and 0.35cm);
      \draw[fill=black] (14,0.00) ellipse (0.06cm and 0.35cm);
      \draw[fill=black] (15,0.00) ellipse (0.06cm and 0.35cm);
      \draw[fill=black] (16,0.00) ellipse (0.06cm and 0.35cm);
      \draw[fill=black] (17,0.00) ellipse (0.06cm and 0.35cm);
      \draw[fill=black] (18,0.00) ellipse (0.06cm and 0.35cm);
      \draw[fill=black] (19,0.00) ellipse (0.06cm and 0.35cm);
      \draw[fill=black] (20,0.00) ellipse (0.06cm and 0.35cm);
      \draw[fill=black] (21,0.00) ellipse (0.06cm and 0.35cm);
      \draw[fill=black] (22,0.00) ellipse (0.06cm and 0.35cm);
      \draw[fill=black] (23,0.00) ellipse (0.06cm and 0.35cm);
      \draw[fill=black] (24,0.00) ellipse (0.06cm and 0.35cm);
      \draw[fill=black] (25,0.00) ellipse (0.06cm and 0.35cm);
      \draw[fill=black] (26,0.00) ellipse (0.06cm and 0.35cm);
      \draw[fill=black] (27,0.00) ellipse (0.06cm and 0.35cm);
      \draw[fill=black] (28,0.00) ellipse (0.06cm and 0.35cm);
      \draw[fill=black] (29,0.00) ellipse (0.06cm and 0.35cm);
      \draw[fill=black] (30,0.00) ellipse (0.06cm and 0.35cm);
      \draw[fill=black] (4,42.20) ellipse (0.06cm and 0.35cm);
      \draw[fill=black] (5,28.50) ellipse (0.06cm and 0.35cm);
      \draw[fill=black] (6,13.08) ellipse (0.06cm and 0.35cm);
      \draw[fill=black] (7,4.26) ellipse (0.06cm and 0.35cm);
      \draw[fill=black] (8,1.34) ellipse (0.06cm and 0.35cm);
      \draw[fill=black] (9,0.45) ellipse (0.06cm and 0.35cm);
      \draw[fill=black] (10,0.09) ellipse (0.06cm and 0.35cm);
      \draw[fill=black] (11,0.00) ellipse (0.06cm and 0.35cm);
      \draw[fill=black] (12,0.00) ellipse (0.06cm and 0.35cm);
      \draw[fill=black] (13,0.00) ellipse (0.06cm and 0.35cm);
      \draw[fill=black] (14,0.00) ellipse (0.06cm and 0.35cm);
      \draw[fill=black] (15,0.00) ellipse (0.06cm and 0.35cm);
      \draw[fill=black] (16,0.00) ellipse (0.06cm and 0.35cm);
      \draw[fill=black] (17,0.00) ellipse (0.06cm and 0.35cm);
      \draw[fill=black] (18,0.00) ellipse (0.06cm and 0.35cm);
      \draw[fill=black] (19,0.00) ellipse (0.06cm and 0.35cm);
      \draw[fill=black] (20,0.00) ellipse (0.06cm and 0.35cm);
      \draw[fill=black] (21,0.00) ellipse (0.06cm and 0.35cm);
      \draw[fill=black] (22,0.00) ellipse (0.06cm and 0.35cm);
      \draw[fill=black] (23,0.00) ellipse (0.06cm and 0.35cm);
      \draw[fill=black] (24,0.00) ellipse (0.06cm and 0.35cm);
      \draw[fill=black] (25,0.00) ellipse (0.06cm and 0.35cm);
      \draw[fill=black] (26,0.00) ellipse (0.06cm and 0.35cm);
      \draw[fill=black] (27,0.00) ellipse (0.06cm and 0.35cm);
      \draw[fill=black] (28,0.00) ellipse (0.06cm and 0.35cm);
      \draw[fill=black] (29,0.00) ellipse (0.06cm and 0.35cm);
      \draw[fill=black] (30,0.00) ellipse (0.06cm and 0.35cm);
      \draw[fill=black] (4,78.94) ellipse (0.06cm and 0.35cm);
      \draw[fill=black] (5,63.28) ellipse (0.06cm and 0.35cm);
      \draw[fill=black] (6,50.65) ellipse (0.06cm and 0.35cm);
      \draw[fill=black] (7,39.90) ellipse (0.06cm and 0.35cm);
      \draw[fill=black] (8,32.53) ellipse (0.06cm and 0.35cm);
      \draw[fill=black] (9,21.29) ellipse (0.06cm and 0.35cm);
      \draw[fill=black] (10,13.58) ellipse (0.06cm and 0.35cm);
      \draw[fill=black] (11,8.30) ellipse (0.06cm and 0.35cm);
      \draw[fill=black] (12,4.58) ellipse (0.06cm and 0.35cm);
      \draw[fill=black] (13,2.41) ellipse (0.06cm and 0.35cm);
      \draw[fill=black] (14,1.19) ellipse (0.06cm and 0.35cm);
      \draw[fill=black] (15,0.67) ellipse (0.06cm and 0.35cm);
      \draw[fill=black] (16,0.33) ellipse (0.06cm and 0.35cm);
      \draw[fill=black] (17,0.15) ellipse (0.06cm and 0.35cm);
      \draw[fill=black] (18,0.06) ellipse (0.06cm and 0.35cm);
      \draw[fill=black] (19,0.02) ellipse (0.06cm and 0.35cm);
      \draw[fill=black] (20,0.00) ellipse (0.06cm and 0.35cm);
      \draw[fill=black] (21,0.00) ellipse (0.06cm and 0.35cm);
      \draw[fill=black] (22,0.00) ellipse (0.06cm and 0.35cm);
      \draw[fill=black] (23,0.00) ellipse (0.06cm and 0.35cm);
      \draw[fill=black] (24,0.00) ellipse (0.06cm and 0.35cm);
      \draw[fill=black] (25,0.00) ellipse (0.06cm and 0.35cm);
      \draw[fill=black] (26,0.00) ellipse (0.06cm and 0.35cm);
      \draw[fill=black] (27,0.00) ellipse (0.06cm and 0.35cm);
      \draw[fill=black] (28,0.00) ellipse (0.06cm and 0.35cm);
      \draw[fill=black] (29,0.00) ellipse (0.06cm and 0.35cm);
      \draw[fill=black] (30,0.00) ellipse (0.06cm and 0.35cm);
      \draw[fill=black] (4,82.66) ellipse (0.06cm and 0.35cm);
      \draw[fill=black] (5,84.66) ellipse (0.06cm and 0.35cm);
      \draw[fill=black] (6,73.78) ellipse (0.06cm and 0.35cm);
      \draw[fill=black] (7,62.74) ellipse (0.06cm and 0.35cm);
      \draw[fill=black] (8,52.93) ellipse (0.06cm and 0.35cm);
      \draw[fill=black] (9,46.04) ellipse (0.06cm and 0.35cm);
      \draw[fill=black] (10,39.04) ellipse (0.06cm and 0.35cm);
      \draw[fill=black] (11,29.16) ellipse (0.06cm and 0.35cm);
      \draw[fill=black] (12,25.43) ellipse (0.06cm and 0.35cm);
      \draw[fill=black] (13,18.54) ellipse (0.06cm and 0.35cm);
      \draw[fill=black] (14,13.38) ellipse (0.06cm and 0.35cm);
      \draw[fill=black] (15,10.10) ellipse (0.06cm and 0.35cm);
      \draw[fill=black] (16,6.23) ellipse (0.06cm and 0.35cm);
      \draw[fill=black] (17,4.70) ellipse (0.06cm and 0.35cm);
      \draw[fill=black] (18,2.87) ellipse (0.06cm and 0.35cm);
      \draw[fill=black] (19,1.90) ellipse (0.06cm and 0.35cm);
      \draw[fill=black] (20,1.23) ellipse (0.06cm and 0.35cm);
      \draw[fill=black] (21,0.77) ellipse (0.06cm and 0.35cm);
      \draw[fill=black] (22,0.44) ellipse (0.06cm and 0.35cm);
      \draw[fill=black] (23,0.30) ellipse (0.06cm and 0.35cm);
      \draw[fill=black] (24,0.18) ellipse (0.06cm and 0.35cm);
      \draw[fill=black] (25,0.08) ellipse (0.06cm and 0.35cm);
      \draw[fill=black] (26,0.05) ellipse (0.06cm and 0.35cm);
      \draw[fill=black] (27,0.04) ellipse (0.06cm and 0.35cm);
      \draw[fill=black] (28,0.01) ellipse (0.06cm and 0.35cm);
      \draw[fill=black] (29,0.00) ellipse (0.06cm and 0.35cm);
      \draw[fill=black] (30,0.00) ellipse (0.06cm and 0.35cm);
      \draw[fill=black] (4,86.68) ellipse (0.06cm and 0.35cm);
      \draw[fill=black] (5,86.99) ellipse (0.06cm and 0.35cm);
      \draw[fill=black] (6,88.07) ellipse (0.06cm and 0.35cm);
      \draw[fill=black] (7,78.24) ellipse (0.06cm and 0.35cm);
      \draw[fill=black] (8,70.51) ellipse (0.06cm and 0.35cm);
      \draw[fill=black] (9,62.55) ellipse (0.06cm and 0.35cm);
      \draw[fill=black] (10,56.02) ellipse (0.06cm and 0.35cm);
      \draw[fill=black] (11,49.82) ellipse (0.06cm and 0.35cm);
      \draw[fill=black] (12,43.75) ellipse (0.06cm and 0.35cm);
      \draw[fill=black] (13,38.72) ellipse (0.06cm and 0.35cm);
      \draw[fill=black] (14,30.86) ellipse (0.06cm and 0.35cm);
      \draw[fill=black] (15,27.59) ellipse (0.06cm and 0.35cm);
      \draw[fill=black] (16,22.11) ellipse (0.06cm and 0.35cm);
      \draw[fill=black] (17,17.46) ellipse (0.06cm and 0.35cm);
      \draw[fill=black] (18,13.75) ellipse (0.06cm and 0.35cm);
      \draw[fill=black] (19,10.82) ellipse (0.06cm and 0.35cm);
      \draw[fill=black] (20,8.54) ellipse (0.06cm and 0.35cm);
      \draw[fill=black] (21,6.20) ellipse (0.06cm and 0.35cm);
      \draw[fill=black] (22,4.86) ellipse (0.06cm and 0.35cm);
      \draw[fill=black] (23,3.44) ellipse (0.06cm and 0.35cm);
      \draw[fill=black] (24,2.46) ellipse (0.06cm and 0.35cm);
      \draw[fill=black] (25,1.74) ellipse (0.06cm and 0.35cm);
      \draw[fill=black] (26,1.30) ellipse (0.06cm and 0.35cm);
      \draw[fill=black] (27,0.74) ellipse (0.06cm and 0.35cm);
      \draw[fill=black] (28,0.57) ellipse (0.06cm and 0.35cm);
      \draw[fill=black] (29,0.38) ellipse (0.06cm and 0.35cm);
      \draw[fill=black] (30,0.26) ellipse (0.06cm and 0.35cm);
      \draw[fill=black] (4,90.78) ellipse (0.06cm and 0.35cm);
      \draw[fill=black] (5,89.69) ellipse (0.06cm and 0.35cm);
      \draw[fill=black] (6,89.79) ellipse (0.06cm and 0.35cm);
      \draw[fill=black] (7,82.32) ellipse (0.06cm and 0.35cm);
      \draw[fill=black] (8,81.41) ellipse (0.06cm and 0.35cm);
      \draw[fill=black] (9,74.39) ellipse (0.06cm and 0.35cm);
      \draw[fill=black] (10,68.67) ellipse (0.06cm and 0.35cm);
      \draw[fill=black] (11,62.19) ellipse (0.06cm and 0.35cm);
      \draw[fill=black] (12,57.21) ellipse (0.06cm and 0.35cm);
      \draw[fill=black] (13,51.68) ellipse (0.06cm and 0.35cm);
      \draw[fill=black] (14,46.97) ellipse (0.06cm and 0.35cm);
      \draw[fill=black] (15,42.25) ellipse (0.06cm and 0.35cm);
      \draw[fill=black] (16,35.07) ellipse (0.06cm and 0.35cm);
      \draw[fill=black] (17,32.34) ellipse (0.06cm and 0.35cm);
      \draw[fill=black] (18,26.55) ellipse (0.06cm and 0.35cm);
      \draw[fill=black] (19,24.07) ellipse (0.06cm and 0.35cm);
      \draw[fill=black] (20,20.13) ellipse (0.06cm and 0.35cm);
      \draw[fill=black] (21,16.63) ellipse (0.06cm and 0.35cm);
      \draw[fill=black] (22,14.06) ellipse (0.06cm and 0.35cm);
      \draw[fill=black] (23,11.36) ellipse (0.06cm and 0.35cm);
      \draw[fill=black] (24,9.26) ellipse (0.06cm and 0.35cm);
      \draw[fill=black] (25,7.16) ellipse (0.06cm and 0.35cm);
      \draw[fill=black] (26,5.82) ellipse (0.06cm and 0.35cm);
      \draw[fill=black] (27,4.39) ellipse (0.06cm and 0.35cm);
      \draw[fill=black] (28,3.66) ellipse (0.06cm and 0.35cm);
      \draw[fill=black] (29,2.73) ellipse (0.06cm and 0.35cm);
      \draw[fill=black] (30,2.07) ellipse (0.06cm and 0.35cm);
      \draw[fill=black] (4,90.97) ellipse (0.06cm and 0.35cm);
      \draw[fill=black] (5,91.98) ellipse (0.06cm and 0.35cm);
      \draw[fill=black] (6,91.42) ellipse (0.06cm and 0.35cm);
      \draw[fill=black] (7,91.71) ellipse (0.06cm and 0.35cm);
      \draw[fill=black] (8,84.73) ellipse (0.06cm and 0.35cm);
      \draw[fill=black] (9,85.64) ellipse (0.06cm and 0.35cm);
      \draw[fill=black] (10,79.17) ellipse (0.06cm and 0.35cm);
      \draw[fill=black] (11,72.21) ellipse (0.06cm and 0.35cm);
      \draw[fill=black] (12,67.33) ellipse (0.06cm and 0.35cm);
      \draw[fill=black] (13,62.01) ellipse (0.06cm and 0.35cm);
      \draw[fill=black] (14,57.60) ellipse (0.06cm and 0.35cm);
      \draw[fill=black] (15,53.20) ellipse (0.06cm and 0.35cm);
      \draw[fill=black] (16,48.79) ellipse (0.06cm and 0.35cm);
      \draw[fill=black] (17,45.27) ellipse (0.06cm and 0.35cm);
      \draw[fill=black] (18,41.24) ellipse (0.06cm and 0.35cm);
      \draw[fill=black] (19,35.36) ellipse (0.06cm and 0.35cm);
      \draw[fill=black] (20,32.97) ellipse (0.06cm and 0.35cm);
      \draw[fill=black] (21,28.35) ellipse (0.06cm and 0.35cm);
      \draw[fill=black] (22,24.19) ellipse (0.06cm and 0.35cm);
      \draw[fill=black] (23,21.98) ellipse (0.06cm and 0.35cm);
      \draw[fill=black] (24,18.90) ellipse (0.06cm and 0.35cm);
      \draw[fill=black] (25,16.55) ellipse (0.06cm and 0.35cm);
      \draw[fill=black] (26,13.54) ellipse (0.06cm and 0.35cm);
      \draw[fill=black] (27,11.60) ellipse (0.06cm and 0.35cm);
      \draw[fill=black] (28,9.95) ellipse (0.06cm and 0.35cm);
      \draw[fill=black] (29,7.94) ellipse (0.06cm and 0.35cm);
      \draw[fill=black] (30,6.68) ellipse (0.06cm and 0.35cm);
      \draw[fill=black] (4,90.96) ellipse (0.06cm and 0.35cm);
      \draw[fill=black] (5,91.91) ellipse (0.06cm and 0.35cm);
      \draw[fill=black] (6,93.04) ellipse (0.06cm and 0.35cm);
      \draw[fill=black] (7,92.71) ellipse (0.06cm and 0.35cm);
      \draw[fill=black] (8,86.30) ellipse (0.06cm and 0.35cm);
      \draw[fill=black] (9,86.76) ellipse (0.06cm and 0.35cm);
      \draw[fill=black] (10,81.25) ellipse (0.06cm and 0.35cm);
      \draw[fill=black] (11,81.24) ellipse (0.06cm and 0.35cm);
      \draw[fill=black] (12,75.94) ellipse (0.06cm and 0.35cm);
      \draw[fill=black] (13,71.44) ellipse (0.06cm and 0.35cm);
      \draw[fill=black] (14,66.17) ellipse (0.06cm and 0.35cm);
      \draw[fill=black] (15,61.81) ellipse (0.06cm and 0.35cm);
      \draw[fill=black] (16,58.22) ellipse (0.06cm and 0.35cm);
      \draw[fill=black] (17,53.70) ellipse (0.06cm and 0.35cm);
      \draw[fill=black] (18,50.09) ellipse (0.06cm and 0.35cm);
      \draw[fill=black] (19,47.10) ellipse (0.06cm and 0.35cm);
      \draw[fill=black] (20,44.31) ellipse (0.06cm and 0.35cm);
      \draw[fill=black] (21,38.28) ellipse (0.06cm and 0.35cm);
      \draw[fill=black] (22,35.41) ellipse (0.06cm and 0.35cm);
      \draw[fill=black] (23,31.91) ellipse (0.06cm and 0.35cm);
      \draw[fill=black] (24,29.50) ellipse (0.06cm and 0.35cm);
      \draw[fill=black] (25,25.54) ellipse (0.06cm and 0.35cm);
      \draw[fill=black] (26,23.87) ellipse (0.06cm and 0.35cm);
      \draw[fill=black] (27,20.88) ellipse (0.06cm and 0.35cm);
      \draw[fill=black] (28,18.17) ellipse (0.06cm and 0.35cm);
      \draw[fill=black] (29,15.99) ellipse (0.06cm and 0.35cm);
      \draw[fill=black] (30,13.85) ellipse (0.06cm and 0.35cm);
      \draw[fill=black] (4,95.16) ellipse (0.06cm and 0.35cm);
      \draw[fill=black] (5,94.50) ellipse (0.06cm and 0.35cm);
      \draw[fill=black] (6,94.83) ellipse (0.06cm and 0.35cm);
      \draw[fill=black] (7,93.99) ellipse (0.06cm and 0.35cm);
      \draw[fill=black] (8,93.60) ellipse (0.06cm and 0.35cm);
      \draw[fill=black] (9,88.01) ellipse (0.06cm and 0.35cm);
      \draw[fill=black] (10,88.00) ellipse (0.06cm and 0.35cm);
      \draw[fill=black] (11,83.71) ellipse (0.06cm and 0.35cm);
      \draw[fill=black] (12,83.22) ellipse (0.06cm and 0.35cm);
      \draw[fill=black] (13,78.23) ellipse (0.06cm and 0.35cm);
      \draw[fill=black] (14,74.28) ellipse (0.06cm and 0.35cm);
      \draw[fill=black] (15,70.20) ellipse (0.06cm and 0.35cm);
      \draw[fill=black] (16,65.33) ellipse (0.06cm and 0.35cm);
      \draw[fill=black] (17,61.92) ellipse (0.06cm and 0.35cm);
      \draw[fill=black] (18,58.59) ellipse (0.06cm and 0.35cm);
      \draw[fill=black] (19,54.93) ellipse (0.06cm and 0.35cm);
      \draw[fill=black] (20,51.97) ellipse (0.06cm and 0.35cm);
      \draw[fill=black] (21,48.87) ellipse (0.06cm and 0.35cm);
      \draw[fill=black] (22,45.83) ellipse (0.06cm and 0.35cm);
      \draw[fill=black] (23,43.04) ellipse (0.06cm and 0.35cm);
      \draw[fill=black] (24,38.11) ellipse (0.06cm and 0.35cm);
      \draw[fill=black] (25,35.77) ellipse (0.06cm and 0.35cm);
      \draw[fill=black] (26,31.81) ellipse (0.06cm and 0.35cm);
      \draw[fill=black] (27,30.20) ellipse (0.06cm and 0.35cm);
      \draw[fill=black] (28,26.52) ellipse (0.06cm and 0.35cm);
      \draw[fill=black] (29,23.41) ellipse (0.06cm and 0.35cm);
      \draw[fill=black] (30,22.61) ellipse (0.06cm and 0.35cm);
      \draw[fill=black] (4,95.19) ellipse (0.06cm and 0.35cm);
      \draw[fill=black] (5,97.28) ellipse (0.06cm and 0.35cm);
      \draw[fill=black] (6,96.48) ellipse (0.06cm and 0.35cm);
      \draw[fill=black] (7,96.25) ellipse (0.06cm and 0.35cm);
      \draw[fill=black] (8,96.41) ellipse (0.06cm and 0.35cm);
      \draw[fill=black] (9,96.35) ellipse (0.06cm and 0.35cm);
      \draw[fill=black] (10,96.66) ellipse (0.06cm and 0.35cm);
      \draw[fill=black] (11,93.83) ellipse (0.06cm and 0.35cm);
      \draw[fill=black] (12,94.17) ellipse (0.06cm and 0.35cm);
      \draw[fill=black] (13,93.78) ellipse (0.06cm and 0.35cm);
      \draw[fill=black] (14,90.77) ellipse (0.06cm and 0.35cm);
      \draw[fill=black] (15,90.56) ellipse (0.06cm and 0.35cm);
      \draw[fill=black] (16,87.52) ellipse (0.06cm and 0.35cm);
      \draw[fill=black] (17,87.42) ellipse (0.06cm and 0.35cm);
      \draw[fill=black] (18,85.03) ellipse (0.06cm and 0.35cm);
      \draw[fill=black] (19,85.31) ellipse (0.06cm and 0.35cm);
      \draw[fill=black] (20,82.04) ellipse (0.06cm and 0.35cm);
      \draw[fill=black] (21,82.54) ellipse (0.06cm and 0.35cm);
      \draw[fill=black] (22,79.95) ellipse (0.06cm and 0.35cm);
      \draw[fill=black] (23,77.35) ellipse (0.06cm and 0.35cm);
      \draw[fill=black] (24,77.06) ellipse (0.06cm and 0.35cm);
      \draw[fill=black] (25,75.06) ellipse (0.06cm and 0.35cm);
      \draw[fill=black] (26,71.86) ellipse (0.06cm and 0.35cm);
      \draw[fill=black] (27,69.93) ellipse (0.06cm and 0.35cm);
      \draw[fill=black] (28,67.63) ellipse (0.06cm and 0.35cm);
      \draw[fill=black] (29,65.77) ellipse (0.06cm and 0.35cm);
      \draw[fill=black] (30,63.60) ellipse (0.06cm and 0.35cm);
      \draw[fill=black] (4,95.28) ellipse (0.06cm and 0.35cm);
      \draw[fill=black] (5,97.17) ellipse (0.06cm and 0.35cm);
      \draw[fill=black] (6,98.17) ellipse (0.06cm and 0.35cm);
      \draw[fill=black] (7,97.52) ellipse (0.06cm and 0.35cm);
      \draw[fill=black] (8,98.02) ellipse (0.06cm and 0.35cm);
      \draw[fill=black] (9,97.96) ellipse (0.06cm and 0.35cm);
      \draw[fill=black] (10,98.35) ellipse (0.06cm and 0.35cm);
      \draw[fill=black] (11,98.10) ellipse (0.06cm and 0.35cm);
      \draw[fill=black] (12,98.07) ellipse (0.06cm and 0.35cm);
      \draw[fill=black] (13,98.13) ellipse (0.06cm and 0.35cm);
      \draw[fill=black] (14,96.15) ellipse (0.06cm and 0.35cm);
      \draw[fill=black] (15,96.17) ellipse (0.06cm and 0.35cm);
      \draw[fill=black] (16,96.28) ellipse (0.06cm and 0.35cm);
      \draw[fill=black] (17,96.01) ellipse (0.06cm and 0.35cm);
      \draw[fill=black] (18,94.39) ellipse (0.06cm and 0.35cm);
      \draw[fill=black] (19,94.37) ellipse (0.06cm and 0.35cm);
      \draw[fill=black] (20,94.01) ellipse (0.06cm and 0.35cm);
      \draw[fill=black] (21,92.66) ellipse (0.06cm and 0.35cm);
      \draw[fill=black] (22,92.62) ellipse (0.06cm and 0.35cm);
      \draw[fill=black] (23,90.83) ellipse (0.06cm and 0.35cm);
      \draw[fill=black] (24,90.53) ellipse (0.06cm and 0.35cm);
      \draw[fill=black] (25,88.95) ellipse (0.06cm and 0.35cm);
      \draw[fill=black] (26,88.64) ellipse (0.06cm and 0.35cm);
      \draw[fill=black] (27,87.17) ellipse (0.06cm and 0.35cm);
      \draw[fill=black] (28,87.17) ellipse (0.06cm and 0.35cm);
      \draw[fill=black] (29,85.39) ellipse (0.06cm and 0.35cm);
      \draw[fill=black] (30,85.86) ellipse (0.06cm and 0.35cm);
      \draw[fill=black] (4,95.26) ellipse (0.06cm and 0.35cm);
      \draw[fill=black] (5,97.28) ellipse (0.06cm and 0.35cm);
      \draw[fill=black] (6,98.13) ellipse (0.06cm and 0.35cm);
      \draw[fill=black] (7,98.82) ellipse (0.06cm and 0.35cm);
      \draw[fill=black] (8,99.07) ellipse (0.06cm and 0.35cm);
      \draw[fill=black] (9,99.25) ellipse (0.06cm and 0.35cm);
      \draw[fill=black] (10,99.44) ellipse (0.06cm and 0.35cm);
      \draw[fill=black] (11,99.50) ellipse (0.06cm and 0.35cm);
      \draw[fill=black] (12,99.62) ellipse (0.06cm and 0.35cm);
      \draw[fill=black] (13,99.70) ellipse (0.06cm and 0.35cm);
      \draw[fill=black] (14,99.72) ellipse (0.06cm and 0.35cm);
      \draw[fill=black] (15,99.75) ellipse (0.06cm and 0.35cm);
      \draw[fill=black] (16,99.76) ellipse (0.06cm and 0.35cm);
      \draw[fill=black] (17,99.83) ellipse (0.06cm and 0.35cm);
      \draw[fill=black] (18,99.82) ellipse (0.06cm and 0.35cm);
      \draw[fill=black] (19,99.88) ellipse (0.06cm and 0.35cm);
      \draw[fill=black] (20,99.86) ellipse (0.06cm and 0.35cm);
      \draw[fill=black] (21,99.86) ellipse (0.06cm and 0.35cm);
      \draw[fill=black] (22,99.87) ellipse (0.06cm and 0.35cm);
      \draw[fill=black] (23,99.89) ellipse (0.06cm and 0.35cm);
      \draw[fill=black] (24,99.92) ellipse (0.06cm and 0.35cm);
      \draw[fill=black] (25,99.92) ellipse (0.06cm and 0.35cm);
      \draw[fill=black] (26,99.92) ellipse (0.06cm and 0.35cm);
      \draw[fill=black] (27,99.94) ellipse (0.06cm and 0.35cm);
      \draw[fill=black] (28,99.94) ellipse (0.06cm and 0.35cm);
      \draw[fill=black] (29,99.90) ellipse (0.06cm and 0.35cm);
      \draw[fill=black] (30,99.88) ellipse (0.06cm and 0.35cm);
      \node at (5,5) {\scriptsize$\ell=1$};
      \node at (7,7) {\scriptsize$\ell=2$};
      \node at (11,11) {\scriptsize$\ell=4$};
      \node at (14,17) {\scriptsize$\ell=6$};
      \node at (15,25) {\scriptsize$\ell=8$};
      \node at (18,30) {\scriptsize$\ell=10$};
      \node at (19,37.5) {\scriptsize$\ell=12$};
      \node at (20.5,42) {\scriptsize$\ell=14$};
      \node at (20.5,54) {\scriptsize$\ell=16$};
      \node at (20.5,68) {$\vdots$};
      \node at (20.5,78) {\scriptsize$\ell=30$};

      \node at (27,82) {\scriptsize$\ell=50$};
      \node at (27,95) {\scriptsize monomial};
    \end{tikzpicture} 
    \caption{The graphs of ratios of vanishing cup product}
    \label{fig:grf}
  \end{center}
\end{figure}
The $x$-axis means the degree $n$ of $G(\underline{x})$ and the $y$-axis means the ratio of vanishing cup product among the 50,000 pairs. When there is no limitation of number of monomial terms (the bottom graph), the ratio of vanishing cup product converges to zero as the degree $n$ grows. However, we see that the vanishing ratio grows as the $\ell$ increases (i.e. the number of monomial terms of $U_0, U_1$ becomes smaller). Moreover, in the top graph case ($U_0$ and $U_1$ themselves are monomials), it converges to 100 as the degree $n$ grows. Therefore, for simplification of computation of the triple Massey product (we need vanishing of the cup product), we choose $U_0, U_1$ so that their number of monomial terms is less than or equal to three.

\begin{algorithm}[H]\label{Alg1}
  \DontPrintSemicolon
  \KwIn{Polynomial $G(\underline x)$}
  \KwOut{Homogeneous polynomials $U_0,U_1,U_2$ such that $U_0U_1\in J_G$ and $U_1U_2\in J_G$}
  \Begin{
    \Repeat{$U_0U_1\in J_G$ and $U_1U_2\in J_G$}{
      $U_0\leftarrow$ homogeneous random polynomial in $\Bbbk[\underline x]$ with degree=$2\deg G-3$ and number of terms $\leq3$\;
      $U_1\leftarrow$ homogeneous random polynomial in $\Bbbk[\underline x]$ with degree=$\deg G-3$ and number of terms $\leq3$\;
      $U_2\leftarrow$ homogeneous random polynomial in $\Bbbk[\underline x]$ with degree=$2\deg G-3$ and number of terms $\leq3$\;
      }
    \Return{$U_0,U_1,U_2$}
  }
  \caption{Making $U_i$ corresponds to $w_i=[W_i]$ for triple massey product randomly.}
\end{algorithm}
The following examples are random polynomials which came from Algorithm \ref{Alg1} where $G(\underline x)=x_0^5+x_1^5+x_2^5$:
\begin{align}
U_0&=x_0^3x_1^4+x_1^5x_2^2 & U_1&=\tfrac{1}{4}x_2^2 & U_2&=\tfrac{1}{3}x_0^4x_1x_2^2,\label{Example1}\\
U_0&=-\tfrac{1}{6}x_0^3x_1^2x_2^2& U_1&=x_2^2 & U_2&=\tfrac{2}{9}x_0^4x_2^3.\label{Example2}
\end{align}\par
Next, let us calculate the triple Massey product. Consider the following algorithm:

\begin{algorithm}[H]\label{Alg2}
  \DontPrintSemicolon
  \KwIn{Polynomials $G,U_0,U_1,U_2$ such that $U_0U_1\in J_G$ and $U_1U_2\in J_G$}
  \KwOut{Triple massey product $\langle U_0,U_1,U_2\rangle$}
  \Begin{
    \tcc{Calculate A,B in Theorem \ref{mt}}
    $R_i^{(01)}\leftarrow$Polynomials such that $U_0U_1=\sum R_i^{(01)}G_i$\;
    $R_i^{(12)}\leftarrow$Polynomials such that $U_1U_2=\sum R_i^{(12)}G_i$\;
    $A,B\leftarrow$ By Theorem \ref{mt}\;
    \tcc{Make $(\Bbbk[\underline x]/J)_{6\deg G-9}$ for Calculation.}
    \nl\label{line2}$M\leftarrow$Ordered set of all of mononmials whose degree is $6\deg G-9$\;
    $N\leftarrow$Set of all af monomials whosee degree is $4\deg G-7$\;
    $L\leftarrow\{G_0^2,G_1^2,G_2^2\}$\;
    \nl\label{line3}$J\leftarrow\emptyset$\;
    \nl\label{line4}\For{$n\in N$}{
      \For{$\ell\in L$}{
        $J\leftarrow J\cup\{n\ell\}$
      }
    }
    \nl\label{line1}$K\leftarrow$ $\Bbbk$-vector space dimension=$|M|$\;
    \nl\label{line5}$J_V\leftarrow$ Span(vector representations of elements in $J$)\;
    $K/J\leftarrow$ Qutient space $K/J_V$\tcc*[r]{1 dimensional vectorspace}
    $\langle U_0,U_1,U_2\rangle\leftarrow$(Image of $\deg G\cdot(A-B)$ in $K/J$)/(Image of $\mathrm{Det}_G(\underline x)$ in $K/J$)\;
    \Return{$\langle U_0,U_1,U_2\rangle$}
  }
  \caption{Calculate $\langle U_0,U_1,U_2\rangle$.}
\end{algorithm}

\noindent The main step of Algorithm \ref{Alg2} is to realize $(\Bbbk[\underline x]/J_G)_{6\deg G-9}$ in a computer. 
We will use a finite dimensional $\Bbbk$-vector space with the standard basis in order to handle the polynomial ring in a computer\footnote{The program \textsf{Sagemath} has a function which calculates a quotient space of a finite dimensional vector space.}. We take $\Bbbk$-vecotr space $K$ in \textsf{line \ref{line1}} with standard basis $\{e_1,\dots,e_s\}$ which corresponds to $M=\{m_1,\dots,m_s\}$ in \textsf{line \ref{line2}}. Then we take $J$ as the basis of $J_{6\deg G-9}$ (\textsf{line \ref{line3}$\sim$\ref{line4}}), and make $J_V$ in \textsf{line \ref{line5}} which corresponds to $J$ under the correspondence between $\{e_1,\dots,e_s\}$ and $\{m_1,\dots,m_s\}$. 
This enables us to compute the quotient space and the remaining calculation using a computer.\par

As a result of Algorithm \ref{Alg2}, we provide two numerical examples with zero (\ref{Example1}) and nonzero (\ref{Example2}) Massey triple products:
\begin{displaymath}
\begin{aligned}
\textrm{(\ref{Example1})}&\quad
\left\{\begin{array}{l}
\quad U_0U_1=0\cdot G_0+\tfrac{1}{20}x_0^3x_2^2\cdot G_1+\tfrac{1}{20}x_1^5\cdot G_2,\\
\quad U_1U_2=0\cdot G_0+0\cdot G_1+\tfrac{1}{60}x_0^4x_1\cdot G_2,\\
\quad A-B=\tfrac{5}{01}x_0^9x_1^6x_2^6-\tfrac{5}{01}x_0^4x_1^11x_2^6+\tfrac{5}{01}x_0^7x_1^5x_2^9,\\
\quad \deg G\cdot(A-B)=0\cdot\mathrm{Det}_G(\underline x)+\tfrac{1}{01}x_0x_1^6x_2^6\cdot G_0^2-\tfrac{1}{01}x_0^4x_1^3x_2^6\cdot G_1^2+\tfrac{1}{01}x_0^7x_1^5x_2\cdot G_2^2,\\
\quad \langle U_0,U_1,U_2\rangle=0.
\end{array}\right.\\
\textrm{(\ref{Example2})}&\quad
\left\{\begin{array}{l}
\quad U_0U_1=0\cdot G_0+0\cdot G_1-\tfrac{1}{30}x_0^3x_1^2\cdot G_2,\\
\quad U_1U_2=0\cdot G_0+0\cdot G_1+\tfrac{2}{45}x_0^4x_2\cdot G_2,\\
\quad A-B=-\tfrac{5}{27}x_0^01x_1^2x_2^7+\tfrac{5}{27}x_0^7x_1^7x_2^7,\\
\quad \deg G\cdot(A-B)= \tfrac{1}{8640000}\cdot\mathrm{Det}_G(\underline x)-\tfrac{1}{27}x_0^4x_1^2x_2^7\cdot G_0^2+0\cdot G_1^2+0\cdot G_2^2,\\
\quad \langle U_0,U_1,U_2\rangle=\tfrac{1}{8640000}.
\end{array}\right.
\end{aligned}
\end{displaymath}
where $\mathrm{Det}_G(\underline x)=8000000x_0^7x_1^7x_2^7$.\par


%
%
%
%
%
%
%
%
%
%
%

%


%


\end{document}